\documentclass[reqno,11pt]{amsart}

\usepackage{amsmath}
\usepackage{amsthm}
\usepackage{amssymb}
\usepackage{amscd}
\usepackage{xypic}
\usepackage{verbatim}
\usepackage[plainpages=false,colorlinks,hyperindex,pdfpagemode=None,bookmarksopen,linkcolor=red,citecolor=blue,urlcolor=blue]{hyperref}
\usepackage{pdflscape}
\usepackage{stmaryrd}
\usepackage{mathabx}
\usepackage{multirow}
\usepackage{appendix}

\swapnumbers
\theoremstyle{plain}
\newtheorem{theorem}[subsubsection]{Theorem}

\newtheorem*{theorem*}{Theorem}
\newtheorem{proposition}[subsubsection]{Proposition}
\newtheorem*{proposition*}{Proposition}
\newtheorem{lemma}[subsubsection]{Lemma}
\newtheorem*{lemma*}{Lemma}
\newtheorem*{fact*}{Fact}
\newtheorem{corollary}[subsubsection]{Corollary}
\newtheorem*{corollary*}{Corollary}

\theoremstyle{definition}
\newtheorem{definition}[subsubsection]{Definition}
\theoremstyle{remark}
\newtheorem*{remark}{Remark}

\renewcommand{\comment}[1] {  }

\DeclareFontFamily{OT1}{rsfs}{}
\DeclareFontShape{OT1}{rsfs}{n}{it}{<-> rsfs10}{}
\DeclareMathAlphabet{\mathscr}{OT1}{rsfs}{n}{it}

\newcommand{\from}{\leftarrow}

\newcommand{\Rep}{\mathrm{Rep}}

\newcommand{\Res}{\mathrm{Res}}

\newcommand{\Z}{\mathbb{Z}}

\newcommand{\CC}{\mathbb{C}}

\newcommand{\PP}{\mathbb{P}}

\newcommand{\RR}{\mathbb{R}}

\newcommand{\QQ}{\mathbb{Q}}

\newcommand{\Hom}{\operatorname{Hom}}

\newcommand{\Aut}{{\operatorname{Aut}}}

\newcommand{\varchi}{\mathcal{X}}
\newcommand{\Gm}{\mathbb{G}_m}
\newcommand{\Ga}{\mathbb{G}_a}

\newcommand{\GL}{\operatorname{GL}}

\newcommand{\PGL}{\operatorname{PGL}}

\newcommand{\SL}{\operatorname{SL}}

\newcommand{\SO}{{\operatorname{SO}}}

\newcommand{\Gal}{\operatorname{Gal}}

\newcommand{\tr}{\operatorname{tr}}

\newcommand{\Vol}{\operatorname{Vol}}
\newcommand{\diag}{{\operatorname{diag}}}

\newcommand{\reg}{{\operatorname{reg}}}

\newcommand{\Id}{\operatorname{Id}}

\newcommand{\Kl}{{\mathcal{K}}}
\newcommand{\pinfty}{{\infty +}}

\begin{document}
\numberwithin{equation}{section}
\setcounter{tocdepth}{1}
\title[Beyond Endoscopy for the relative trace formula I]{Beyond Endoscopy for the relative trace formula I: local theory}
\author{Yiannis Sakellaridis}

\begin{abstract}
For the group $G=\PGL_2$ we prove nonstandard matching and the fundamental lemma between two relative trace formulas: on one hand, the relative trace formula of Jacquet for the quotient $T\backslash G/T$, where $T$ is a nontrivial torus; on the other, the Kuznetsov trace formula with \emph{nonstandard} test functions. The matching is nonstandard in the sense that orbital integrals are related to each other not one-by-one, but via an explicit integral transform. These results will be used in \cite{SaBE2} to compare the corresponding global trace formulas and reprove the celebrated result of Waldspurger \cite{Waldspurger-torus} on toric periods. 
\end{abstract}

\maketitle
\tableofcontents

\section{Introduction}

\subsection{} With the present paper I launch an investigation of new ways to compare trace formulas in the field of automorphic forms, as a means of proving explicit relations between the spectra of two relative trace formulas (RTFs) -- which is translated to relations between periods of automorphic forms. Such relations are predicted, in great generality, by the generalization of the Gross--Prasad--Ichino--Ikeda conjecture \cite{II} which is to appear in my joint work with Venkatesh \cite{SV}, and while the general conjecture is not yet very detailed, many special cases suggest the relevance of the RTF in formulating a detailed conjecture: the conjecture is not really about a \emph{single} pair $(G,H)$ consisting of a group and a subgroup, not even about the quotient space $X=H\backslash G$, but it involves certain \emph{``pure inner forms''} of the pair $(G,H)$ which can  be understood in terms of the algebraic stack $H\backslash G/H$ (equivalently: $X\times X/G$). On the other hand, the relative trace 
formula, as currently being used following the paradigm of endoscopy for the Arthur--Selberg trace formula, seems to have no hope of proving relations in such generality, for reasons that will be explained below. Thus, we need new ways to compare trace formulas, which is what I am doing in the present paper, for two periods about which virtually everything is already known: the Whittaker period and the torus period, both for the group $\PGL_2$. The continuation of the present paper in \cite{SaBE2} will provide a new proof of the celebrated result of Waldspurger on toric periods of automorphic forms.

The topic of interest should not be seen as restricted to the study of periods and hence as something separate from the mainstream Langlands program. Indeed, the relative trace formula should be considered a potential generalization of the Arthur-Selberg trace formula (before stabilization), and the nonstandard comparison performed here is certainly within the spirit of the ambitious ``beyond endoscopy'' program proposed by Langlands \cite{LaBE}. The difference lies in the level of ambition and difficulty: while Langlands wants to filter out only part of the spectrum of a trace formula by taking residues of $L$-functions, in order to detect the image of any chosen functorial lift, here I perform a full comparison of two trace formulas; thus, the spectral content is essentially dictated by the $L$-groups of the pertinant spherical varieties \cite{SV}. On the other hand, many other features of the ``beyond endoscopy'' project are present. In particular, $L$-functions are inserted via nonstandard (not 
compactly supported)
 test functions; this is essential in view of the fact that, say, the Whittaker period and the torus period correspond to different $L$-functions; thus, a comparison of the corresponding RTFs using standard test functions would be impossible.

Moreover, the comparison itself is \emph{nonstandard}: Unlike the paradigm of endoscopy, there is no matching of orbits such that the corresponding orbital integrals be preserved by the matching of functions; instead, matching is accomplished via a certain integral transform on the set of orbital integrals. While the existence of such a transform is more or less predicted by spectral matching, and can in principle be calculated whenever one has enough local information about (generalized) characters, it is important that for the example at hand \emph{we are able to give an explicit formula for it}, in terms of Fourier transforms and birational maps. This is extremely important for the global story: since the comparison does not preserve orbital integrals, one will need to prove some kind of Poisson summation formula in order to obtain an identity between matching trace formulas; clearly, a Poisson summation formula for an arbitrary integral transform is no trivial thing -- it might prove to be a 
reformulation of 
the functional equation of some difficult $L$-function. In my understanding, noone yet has a conceptual reason why such a Poisson summation formula should be provable for the integral transforms that will appear in the ``beyond endoscopy'' project; thus, understanding instances where the transforms between spaces of orbital integrals are tractable is an important task.

\subsection{} Now I summarize the contents of the present paper. Throughout, $G$ denotes the algebraic group $\PGL_2$ over a local field $F$, or the $F$-points of $G$. The goal is to compare the space of orbital integrals for the trace formula corresponding to the quotient $T\backslash G/T$, where $T$ is a split or nonsplit torus, and the space of orbital integrals for the Kuznetsov trace formula -- this is the trace formula for the quotient $(N,\psi)\backslash G/(N,\psi)$, where $N$ is a unipotent subgroup and $\psi$ a nontrivial
character of the $F$-points of $N$.

As explained before, such a comparison should not be possible for reasons related to different special values of $L$-functions appearing in global periods. Therefore, we need to modify the space of test functions -- the standard choice is to take Schwartz functions on $G$, but now our test functions for the Kuznetsov trace formula will have nontrivial asymptotics at infinity. The best way to describe this is to think of the Kuznetsov orbital integrals not as distributions on $G$, but as $G$-invariant hermitian forms on $\mathcal S(N\backslash G ,\psi)$. Then one should replace the standard Schwartz space $\mathcal S(N\backslash G ,\psi)$ by sections which have a certain prescribed behavior at ``infinity'', where ``infinity'' means the partial compactification of $N\backslash G$ by $\PP^1$. I describe these nonstandard sections in \S \ref{ssnonstandard}.

The affine quotient $N\backslash G\sslash N$ is isomorphic to $\mathbb A^1$ (one-dimensional affine space), and so is the quotient $T\backslash G\sslash T$. We will denote both by $\mathcal B$, the ``base'' of our quotient stacks. These quotients rougly parametrize orbits, and with suitable conventions that are explained in sections \ref{sec:babycase}, \ref{sec:torus} and \ref{sec:Kuznetsov} we understand the orbital orbital integrals as densely defined functions on $\mathcal B$. Here ``densely'' means that we only consider regular orbital integrals, where the stabilizers are trivial. Thus, we end up with two spaces of densely defined functions on $\mathcal B$, the space $\mathcal S(\mathcal Z)$ (from orbital integrals for $T\backslash G/T$ equipped with standard test functions) and the space $\mathcal S(\mathcal W)$ (from orbital integrals for the Kuznetsov quotient with nonstandard test functions). 

Clearly (for anyone who has some experience with those trace formulas), these spaces of functions are not even closely related to each other as spaces of functions, that is: there is no orbit-by-orbit matching of the two trace formulas. However, the first main result (Theorem \ref{matching}) is that there is an explicit integral transform which takes one to the other. To state it, let $\mathcal F$ denote usual Fourier transform in one variable (with respect to characters and measures that are described in the text), let $\eta$ be the character of $F^\times$ corresponding to the splitting field of $T$, and let $\iota$ be the following operator on functions on $\mathcal B$:
$$\iota(f) = \frac{\eta(\bullet)}{|\bullet|} f\left(\frac{1}{\bullet}\right) .$$
Define the operator: 
$$\mathcal G = \mathcal F \circ \iota \circ \mathcal F.$$
\begin{theorem*}[``Matching'' \ref{matching}]
 The operator $|\bullet|\mathcal G$ is an isomorphism: $$\mathcal S(\mathcal Z)\to \mathcal S(\mathcal W).$$
\end{theorem*}

The second main result, for $F$ nonarchimedean and $T$ unramified over $\QQ_p$ or $\mathbb F_p((t))$ is a fundamental lemma for elements of the Hecke algebra. Here we need to specify (\S \ref{ssbasic}) ``basic vectors'' $f^0_{\mathcal Z}$, $f^0_{\mathcal W}$ for the spaces $\mathcal S(\mathcal Z)$, $\mathcal S(\mathcal W)$, which are obtained by the orbital integrals of certain $K:=G(\mathfrak o)$-invariant functions ``upstairs''. For the torus trace formula this will be the standard unramified test function, but for the Kuznetsov formula it has to be nonstandard, and tailored in order to produce the correct $L$-value $L(\pi,\frac{1}{2})L(\pi\otimes \eta,\frac{1}{2})$ on the spectral side of the trace formula.

Acting by an element $h$ of the spherical Hecke algebra $\mathcal H(G,K)$ on those functions upstairs, we denote their orbital integrals by $h\star f^0_{\mathcal Z} \in \mathcal S(\mathcal Z)$, $h\star f^0_{\mathcal W} \in \mathcal S(\mathcal W)$. The fundamental lemma states that the above integral transform carries one to the other:
\begin{theorem*}[``Fundamental Lemma'' \ref{FL}]
 For any $h\in \mathcal H(G,K)$ the operator $|\bullet|\mathcal G$ carries $h\star f^0_{\mathcal Z}$ to $h\star f^0_{\mathcal W}$.
\end{theorem*}

\subsection{} The results of this paper will be used in the sequel \cite{SaBE2} in order to reprove the theorem of Waldspurger \cite{Waldspurger-torus} on the Euler factorization of toric periods. This global application is far from a straightforward application of the local results; the first difficulty has to do with the fact that we do not have an orbit-by-orbit matching of trace formulas, and hence the matching of global trace formulas has to be proven by a quite nontrivial application of the Poisson summation formula. Other difficulties have to do with the fact that globally there are conconvergent Euler products, and we need to interpret them by analytic continuation; that is why in section \ref{sec:explicit} we introduce variations of our spaces by a complex parameter $s\in \CC$.

\subsection{} While I am optimistic about the possibility of generalizing such nonstandard comparisons to higher rank (and hence proving new period relations that generalize the results of Waldspurger), I should add a word of caution: The relations that we get here can be seen as reflections at the level of orbital integrals to well-known results ``upstairs'': on one hand, the proof of Waldspurger's results by Jacquet \cite{JaW1} relating orbital integrals for $T\backslash G/T$ to orbital integrals for the same quotient with a split torus; and on the other, the method of Hecke for calculating split torus periods via Fourier coefficients. I explain these relations in the proof of the fundamental lemma \ref{FL}. 

The point if the present paper is that while the work of Jacquet and Hecke does not generalize to all other periods one is interested in, the relations between orbital integrals might generalize. Of course, we will not be able to tell which is the case before studying many more examples and trying to find a pattern for nonstandard comparisons.

\subsection{Acknowledgements} I would like to thank Akshay Venkatesh for pointing my attention to his thesis \cite{V} as a possible source of ideas for attacking the period conjectures of \cite{SV}. I would also like to thank Joseph Bernstein, who taught me the correct way to think about several aspects of the relative trace formula. Finally, it is my pleasure to acknowledge the support of the Institute for Advanced Study via NSF grant DMS-0635607 during the spring semester of 2011, when this work was initiated, and the support of NSF grant DMS-1101471.

\subsection{Notation}\label{ssnotation} Notation is mostly local, redefined in every section. For convenience of the reader, we give an overview of the symbols that are most frequently used (and will be defined in the text, if nonstandard): 

\begin{description}
 \item $F$ is a local, locally compact field, $E$ a quadratic etale extension of $F$. The quadratic character of $F^\times$ associated to $E$ is denoted by $\eta = \eta_{E/F}$. If $F$ is nonarchimedean, we denote by $\mathfrak o$ its ring of integers, by $\varpi$ a uniformizer and by $q$ the order of its residue field. 

 \item We feel free to use the same symbol $Y$ to denote the variety $Y$ and its points over $F$, whenever this causes no confusion. For example, ``functions on $Y$'' means functions on the $F$-points of $Y$. When there is ambiguity, we will be denoting the latter by $Y(F)$.
 
\item We fix a measure $dx$ on $F$ as explained in \S \ref{ssTamagawa}, and a unitary complex character $\psi$ of $F$ with respect to which $dx$ is self-dual.

 \item For $p$-adic groups, the usual notion of ``smooth'' vectors and representations typically gives rise to inductive limits of Fr\'echet spaces. To achieve uniformity with the archimedean case, we describe in appendix \ref{app:almostsmooth} a notion of ``almost smooth'' vectors which gives rise to Fr\'echet space representations. For simplicity, we call these vectors ``smooth'' throughout the rest of the text and treat the archimedean and nonarchimedean cases together whenever possible, but the reader may ignore this and focus on smooth vectors in the traditional sense, replacing the Fr\'echet spaces that we consider with their corresponding limits of Fr\'echet spaces of their smooth vectors. 

 \item The notion of Schwartz functions on an open semialgebraic set of a real or $p$-adic manifold is defined in appendix \ref{app:almostsmooth}. When ``Schwartz function'' appears without specifying on which set, we mean Schwartz function on $F$.

 \item For spaces of smooth functions or sections of line bundles, we use the word ``stalk'' as in \S \ref{ssstalks} of the appendix, that is: an element of the stalk over a closed set is defined modulo Schwartz functions/sections on its complement -- \emph{not} modulo compactly supported functions/sections on its complement. In particular, the germ of a smooth function at a point is completely determined by its derivatives at that point.

 \item The symbols $\mathcal X,\mathcal Z, \mathcal W$ are reserved throughout the text for certain quotient stacks related, respectively, to the ``baby case'' of section \ref{sec:babycase}, the torus quotient of section \ref{sec:torus} or the Kuznetsov quotient of section \ref{sec:Kuznetsov}. Other than that, we define and redefine symbols locally, for instance the letter $X$ usually denotes some homogeneous space, which is changing through the text.

 \item $\mathcal B$ denotes the ``base'' of our quotient stacks, i.e.\ the associated GIT quotient. In fact, throughout our examples we have an isomorphism: $\mathcal B\simeq \Ga$, which in some cases is completely canonical and in others is obtained by some choices which we fix (cf.\ the remarks after Proposition \ref{stackisomorphism}).
For each quotient space that we are considering, the ``base'' $\mathcal B$ has a regular set $\mathcal B^\reg$ (different in each case). We will be using this notation when it is clear which quotient space we are referring to, and the notation $\mathcal B^\reg_{\mathcal X},\mathcal B^\reg_{\mathcal Z}$, etc.\ when we want to indicate the quotient space.

 \item $\mathcal M$ denotes the space of Schwartz measures on the points of a real or $p$-adic variety, as well as on quotient stacks (defined as spaces of coinvariants). For the notions of ``Schwartz'' (measures or functions) in the archimedean case, we point the reader to \cite{AG-Nash}. We explain the natural extension of this to spaces of ``almost smooth'' functions on $p$-adic varieties in appendix \ref{app:almostsmooth}. The reader will not miss any essential point, though, by thinking instead of compactly supported functions (the only mathematical problem being that those are not preserved by Fourier transforms, in the real case; also, notationally they are clumsier because they do not form Fr\'echet spaces).

 \item $\mathcal S$ denotes the space of Schwartz functions on the points of a real or $p$-adic variety, and also the space of functions on $\mathcal B^\reg$ obtained by orbital integrals of the test functions ``upstairs''. 

 \item $O_\bullet$ is used to denote ``regular'' orbital integrals, while $\tilde O_\bullet$ is used to denote invariant distributions supported over the irregular points of the base.
\end{description}

\section{A baby case.} \label{sec:babycase}

\subsection{} Before we work with non-commutative groups, we discuss the baby case of the relative trace formula for the variety $X:=\Ga$ of the group $G:=T:=\Gm$, in order to examine certain integral transforms which will be useful in the sequel. We will also discuss a non-split form of the quotient $X\times X/G$, which although not under the general formalism of the relative trace formula, will provide us with some necessary integral transforms.

\subsection{The split case} We let $V$ denote an 1-dimensional vector space over $F$, and $V^*$ its dual. In some of the calculations below, we will be identifying $V$ and $V^*$ with $F$ under the pairing $\left<x,y\right>=xy$. We let $G=T$ be the group $\Gm$ acting diagonally on $V\times V^*$. In what follows, we will use the symbols $G$ and $T$ interchangeably, choosing in each case the one that is most relevant for the applications of the baby case to later chapters.

The stack-theoretic quotient $V\times V^*/\Gm$ will be denoted by $\mathcal X$; its ``dual'' quotient $V^*\times V/\Gm$ will be denoted by $\mathcal X^*$. 

The pairing between $V$ and $V^*$ induces a canonical identification of the categorical quotient of affine spaces (the ``bases''):
\begin{equation}
\mathcal B:= V\times V^* \sslash T \xrightarrow{\sim} \Ga \xleftarrow{\sim} V^*\times V \sslash T.
\end{equation}

We let $\mathcal B^\reg$, ``the regular part of the base'', denote the complement of zero. Notice that for every $\xi\in \mathcal B^\reg$ the fiber is a single orbit of $T$ (both as a variety and in the sense of $F$-points).

\subsection{Tamagawa measures} \label{ssTamagawa}

If instead of $F$ we were talking about a global field $k$, we would be fixing the measure on its ring of adeles which comes from a globally defined differential form on $\Res_{k/\mathbb Q} \Ga$ and the usual measure on $\mathbb A_{\mathbb Q}$. We factorize this in the standard way locally \cite{Tate}, namely:

If $F=\RR$, the usual Lebesgue measure, if $F=\CC$ the double of the usual Lebesgue measure and if $F$ is nonarchimedean the Haar measure under which the ring of integers has measure equal to the inverse square root of the discriminant (hence, $\le 1$).

\subsection{Measures and coinvariants}

We let $\mathcal S(V\times V^*)$ denote the Fr\'echet space of Schwartz functions on $V\times V^*$. Notice that the usual notion of ``Schwartz space'' in the nonarchimedean case does not correspond to a Fr\'echet space but to a limit of Fr\'echet spaces. We explain in the appendix \ref{app:almostsmooth} how to obtain a Fr\'echet space completion thereof, consisting of functions which have the same decay at infinity as in the archimedean case (faster than the absolute value of any polynomial) and are \emph{almost smooth} instead of smooth. For practical purposes, the difference between the two approaches is unimportant, and the reader can keep the traditional Schwartz space in their mind. The introduction of a Fr\'echet space just creates the convenience of treating the archimedean and nonarchimedean cases simultaneously. As explained in the introduction, we will be just using the word ``smooth'' for what should be ``almost smooth''; also, any statement involving derivatives (other than the zeroth) should 
be considered as void in the nonarchimedean case.

We denote by $\mathcal M(V\times V^*)$ the corresponding Fr\'echet space of Schwartz measures on $V\times V^*$ (i.e.\ products of a Schwartz function with additive Haar measure). A choice of Haar measure on $V\times V^*$ defines an isomorphism: $\mathcal S(V\times V^*)\simeq \mathcal M(V\times V^*)$, but we will not need to fix such an isomorphism except as a convenience for certain calculations. 

We define the \emph{Schwartz space of measures on $\mathcal X$} to be:
\begin{equation}\mathcal M(\mathcal X) := \mathcal M(V\times V^*)_G,
\end{equation}
the coinvariant space $\mathcal M(V\times V^*)_G$. By definition, the $G$-coinvariant space of a Fr\'echet representation $W$ is the quotient by the closed subspace generated by vectors of the form $v-g\cdot v$, $v\in W, g\in G$, or equivalently the universal quotient: $W\to W_G$, with trivial $G$-action on the right, through which every continuous $G$-invariant functional factors. In particular, it is naturally a nuclear Fr\'echet space. The space $\mathcal F(\mathcal X)$ of  $G$-invariant, \emph{tempered} generalized functions on $V\times V^*$ is, tautologically, the dual of $\mathcal M(\mathcal X)$. 

The following is standard: 
\begin{lemma} \label{density0}
 $\mathcal F(\mathcal X)$ is the weak-* closure of the space spanned by those invariant generalized functions which are each supported on a fiber of the map: $\mathcal X\to \mathcal B$. 
\end{lemma}

We will see a strengthening of it in Lemma \ref{density1}.

\subsection{Orbital integrals}\label{ssbabyorbital}

For $\xi\in \mathcal B^\reg$, $\Phi\in\mathcal S(V\times V^*)$ and a Haar measure $dg$ on $G$ we define the orbital integral:
\begin{equation}
O_\xi(\Phi) = \int_{G} \Phi(\tilde\xi\cdot g) dg.
\end{equation}
Here $\tilde \xi$ is any lift of $\xi$ to $V\times V^*$.

We define $\mathcal S(\mathcal X)$ to be the space of functions on $\mathcal B^\reg$ of the form $\xi\mapsto O_\xi(\Phi)$, for $\Phi\in \mathcal S(V\times V^*)$. Throughout the text, when we talk about ``a lift of an element $f\in\mathcal S(\mathcal X)$ to $\mathcal S(V\times V^*)$'' we will implicitly mean a pair consisting of an element $\Phi\in\mathcal S(V\times V^*)$ and a Haar measure on $G$, so that $f$ is obtained by the orbital integrals of $\Phi$. Our first goal is to define (and normalize) a linear map: $\mathcal M(\mathcal X)\to \mathcal S(\mathcal X)$. To do this, we start with the following integration formula:

\begin{lemma}\label{integration-babysplit}
 For any $\Phi\in\mathcal S(V\times V^*)$ with image $f\in\mathcal S(\mathcal X)$, and Haar measures on $V\times V^*$ and $G$, we have:
\begin{equation}
\int_{V\times V^*} \Phi(v,v^*) dv dv^* = \int_{\mathcal B} f(\xi) d\xi,
\end{equation}
where $d\xi$ is an additive Haar measure on $\mathcal B=F$. If we take on $V\times V^*$ the Haar measure corresponding to the differential form $dv\wedge dv^*$ and the standard Haar measure on $F$, where the coordinates $v$ and $v^*$ are defined using a dual basis, $dg$ is the multiplicative Haar measure $|a|^{-1}da$ on $F^\times$, then $d\xi$ is the standard Haar measure on $F$ (\S \ref{ssTamagawa}).
\end{lemma}

If we now \emph{fix} the measure on $\mathcal B$ to be the standard measure on $F$ discussed in \S \ref{ssTamagawa}, we get a map:
\begin{equation}\label{measurestoorbital-baby}
 \mathcal M(\mathcal X)\to \mathcal S(\mathcal X),
\end{equation}
as follows: a choice of compatible measures on $V\times V^*$ and $G$ gives an isomorphism:
$$\mathcal M(V\times V^*)\simeq \mathcal S(V\times V^*)$$
and a map:
$$\mathcal S(V\times V^*)\to \mathcal S(\mathcal X),$$
and it is easy to see that the composition of the two depends only on the chosen measure on $\mathcal B$. For the purpose of calculations later in the chapter we will fix the Haar measures on $V\times V^*$ and $G$ described in Lemma \ref{integration-babysplit}.\footnote{Even fixing a measure on $\mathcal B$ locally is not important, since globally we always have canonical choices of measures (Tamagawa measures); nonetheless it will be helpful for calculations to fix the local maps (\ref{measurestoorbital-baby}).} 

The following strengthening of Lemma \ref{density0} will be a corollary of Proposition \ref{germs-split}. 

\begin{lemma} \label{density1}
 The functionals $O_\xi$, $\xi\in\mathcal B^\reg$ span a weak-* dense subspace of $\mathcal F(\mathcal X)$.
\end{lemma}

 This implies that \emph{the map (\ref{measurestoorbital-baby}) is an isomorphism of vector spaces}. Therefore, we will not be distinguishing from now on between $\mathcal M(\mathcal X)$ and $\mathcal S(\mathcal X)$, and we will endow $\mathcal S(\mathcal X)$ with the Fr\'echet topology induced from this identification.

In appendix \ref{app:cosheaves} we introduce a notion of ``Schwartz cosheaves''. We point the reader there for definitions of restriction, stalks and other notions. By Corollary \ref{corollaryflabby}, $\mathcal S(\mathcal X)$ is the space of global sections of \emph{a flabby} (i.e.\ extension maps are injective -- in fact, closed embeddings) \emph{Schwartz cosheaf} on $\mathcal B$, 
which for simplicity we will also be referring to by the symbol $\mathcal S(\mathcal X)$ when there is no confusion. Via the regular orbital integrals, this cosheaf is identified with a cosheaf of functions on $\mathcal B^\reg$; our purpose is to describe this cosheaf.

\begin{lemma}
 The restriction of the cosheaf $\mathcal S(\mathcal X)$ to $\mathcal B^\reg$ is equal to the cosheaf $\mathcal S(\mathcal B^\reg)$ (Schwartz functions in the usual sense).
\end{lemma}

\begin{proof}
 Over $\mathcal B^\reg$ we have a $G$-isomorphism of the variety $V\times V^*$ with $\mathcal B^\reg \times G$, and therefore $\mathcal X^\reg:= \mathcal X\times_{\mathcal B} \mathcal B^\reg$ is isomorphic to $\mathcal B^\reg$.\footnote{There is clearly some work to be done to establish that isomorphisms of stacks give rise to isomorphisms of their Schwartz spaces, but in each of the cases that we are considering in this paper this is easy to see explicitly.}
\end{proof}

Now we focus our attention on the neighborhood of $0$:

\begin{proposition}\label{germs-split}
 For $\Phi\in\mathcal S(V\times V^*)$ we have, for $\xi\in\mathcal B^\reg$ in a neighborhood of $0$:
\begin{equation}
 O_\xi(\Phi) = - C_1(\xi)\cdot \ln |\xi| + C_2(\xi),
\end{equation}
with $C_1, C_2$ (almost) smooth functions (which can be arbitrary).

Moreover, the distributions: 
\begin{equation} \label{babyipsplit}
 \tilde O_0(\Phi):=  C_1(0)
\end{equation}
and 
\begin{equation}
 \tilde O_u(\Phi):= C_2(0) 
\end{equation}
are a basis for the space of functionals on the fiber of $\mathcal S(\mathcal X)$ over $0\in \mathcal B$, and we have:
\begin{equation}\label{O0}
 \tilde O_0(\Phi) =  \Vol(T(F)_0) \Phi(0),  
\end{equation}
where $\Vol(T(F)_0)$ is the volume of the maximal compact subgroup of $T(F)$ described below,
and:
\begin{eqnarray}\label{Ou}
\tilde O_u(\Phi) = \lim_{s\to 0} \left(\zeta(\left.\Phi\right|_{y=0},s) +  \zeta(\left.\Phi\right|_{x=0},-s) \right) \nonumber =\\
 = \left.\frac{d}{ds}\right|_{s=0} \left(s\zeta(\left.\Phi\right|_{y=0},s) + s\zeta(\left.\Phi\right|_{x=0},s)  \right),
\end{eqnarray}
where $\zeta$ is the Tate integral of a function of one variable against unramified characters.
\end{proposition}

\begin{remark}
 If we set $f(\xi) = O_\xi(\Phi)$ then the distributions $\tilde O_0, \tilde O_u$ have been defined in such a way that they depend only on $f\in \mathcal S(\mathcal X)$ and not on $\Phi$ and the choice of Haar measure on $G$, i.e.\ if we modify $\Phi$ and the Haar measure simultaneously so that the orbital integrals of $\Phi$ continue to give $f$, we get the same values of $\tilde O_0(\Phi), \tilde O_u(\Phi)$. It is therefore meaningful to write: $\tilde O_0(f), \tilde O_u(f)$.

The volume mentioned in the lemma is obtained as follows. Notice that the absolute value gives a canonical short exact sequence:
\begin{equation}\label{seqtori}1 \to T(F)_0 \to T(F) \xrightarrow{\|\bullet\|} \Hom(\varchi^*(T)_F,|F^\times|)\to 1,
\end{equation}
where $T(F)_0$ denotes the maximal compact subgroup of $T(F)$, $\varchi^*(T)_F \simeq \Z$ is the $F$-character group of $T$, and $|F^\times|\subset \RR^\times_+$ denotes the group of absolute values of $F^\times$. The group $\Hom(\varchi(T)_F,|F^\times|)$ is canonically, up to inversion, a subgroup of $\RR_+^\times$ (all of $\RR_+^\times$ in the archimedean case, the group $q^{\mathbb Z}$ in the nonarchimedean case). \emph{We endow it with a Haar measure $d|t|$ that is on average equal to the standard multiplicative measure $t^{-1}dt$ on $\RR_+^\times$.} Hence, in the nonarchimedean case with residual degree $q$, $dt(\{1\}) = \ln q$.  
For a Haar measure $\mu$ on $T$, the disintegration of $\frac{d\mu}{d|t|}$ with respect to the map (\ref{seqtori}) is a Haar measure on $T(F)^0$, and we let:
$$ \Vol(T(F)_0) = \frac{d\mu}{d|t|} (T(F)_0).$$
In particular, in the nonarchimedean case, this is $(\ln q)^{-1}$ times $\mu(T(F)_0)$.

\end{remark}

\begin{proof} It is easy to see that all functions of the form $c_2-c_1\ln|\xi|$ can be obtained as orbital integrals in a neighborhood of zero, with $c_i=C_i(0)$ as claimed. It is then easy to see that the invariant distributions on the fiber of $\mathcal S(V\times V^*)$ over the preimage of $0$ are given by \eqref{O0} and \eqref{Ou}. Thus, by Proposition \ref{propNakayama}, the orbital integrals of all elements of $\mathcal S(V\times V^*)$ are of the form $C_2(\xi) - C_1(\xi)\ln|\xi|$ with $C_i$ (almost) smooth functions.

This proves the lemma, and also Lemma \ref{density1}.
\end{proof}

\subsection{Fourier transform}
We choose a character $\psi: F\to \CC^\times$ to identify the space $V^*\times V$ with the Pontryagin dual of $V\times V^*$; we choose it in such a way that the measure \ref{ssTamagawa} on $F$ is self-dual with respect to Fourier transform. Notice that, when working with a global field, adele class characters can be factorized as products of such characters.

In this paper ``Fourier transform'' stands for the usual, schoolbook Fourier transform without any modifications to preserve equivariance, which in one variable $v\in V$ reads: $\hat f(v^*)= \int_F f(v) \psi^{-1}(\left<v,v^*\right>) dv$. It will be denoted both by $\hat~$ and also by the letter $\mathcal F$.

 Hence we have:
$$\hat{\hat\Phi}(x,y) = \Phi(-x,-y).$$

Since Fourier transform in one variable satisfies: $$\mathcal F(f(a\bullet)) (y)= \frac{1}{|a|} \mathcal F(f) \left(\frac{y}{a}\right),$$ it is clear that Fourier transform on $V\times V^\times$ is equivariant with respect to the action of $G$ on $V\times V^*$ and on its dual, and therefore descends to an isomorphism:
\begin{equation}
 \mathcal S(\mathcal X) \xrightarrow{\sim} \mathcal S(\mathcal X^*).
\end{equation}

It is therefore natural to ask how it transforms orbital integrals, or in other words: For each $\Phi\in \mathcal S(V\times V^*)$, express the function $\mathcal B^\reg\ni \xi\mapsto O_\xi(\Phi)$ in terms of the function $\xi \mapsto O_{\xi}(\hat\Phi)$.

\subsection{The integral transform $\mathcal G$.}

\begin{definition} We let $\mathcal G$ denote the transform which maps $f\in \mathcal S(\mathcal X)$ to the Fourier transform of the (tempered) function $y \mapsto \frac{\eta(y)}{|y|} \hat f\left(\frac{1}{y}\right)$, that is:
 \begin{equation} \label{defG}
  \mathcal G = \mathcal F\circ \iota \circ\mathcal F, 
 \end{equation}
where $\iota(f) = \frac{\eta(\bullet)}{|\bullet|} f\left(\frac{1}{\bullet}\right)$.
\end{definition}
Of course, in the split case that we are currently discussing we have $\eta = 1$. 

The following lemma shows, in particular, that the function $y \mapsto \frac{\eta(y)}{|y|} \hat f\left(\frac{1}{y}\right)$ is in $L^2(\mathcal B)$ and hence its Fourier transform makes sense as a function.

\begin{lemma}\label{asympt}
 The Fourier transform of any $f\in \mathcal S(\mathcal X)$ has the property that:
$$\lim_{x\to\infty} |x| \hat f(x)$$ exists.
\end{lemma}

For the following proof, and later use, we normalize the action of $F^\times$ on functions on $F$ in such a way that it is unitary (with respect to the $L^2(F)$-inner product):
\begin{equation}\label{unitaryaction}
 (a\cdot f)(x) = |a|^\frac{1}{2} f(ax).
\end{equation}
Hence, Fourier  transform is anti-equivariant with respect to this action.

\begin{proof}

It is easy to see that the function $l:x \mapsto \ln|x|$ becomes smooth by application of the operator $(\Id - |a|^{-\frac{1}{2}} a\cdot)$, for all $a\in F^\times$. Therefore, its Fourier transform (considered as a tempered generalized function) will become rapidly decaying by application of the operator $(\Id - |a|^{-\frac{1}{2}} a^{-1}\cdot)$. Hence:
$$\hat l(x) = \frac{c}{|x|} + h_1(x)$$
in a neighborhood of infinity, for some constant $c$ and some Schwartz function $h_1(x)$.

Thus, the Fourier transform of an element of $\mathcal S(\mathcal X)$ will be of the form: 
$$h(x) =  h_2(x) \star \hat l(x) + h_3(x)$$
in a neighborhood of infinity, where $h_i$ are Schwartz functions and $\star$ denotes convolution. It is easy to see that for such a function the limit: $\lim_{|x|\to\infty} |x| h(x)$ exists.
\end{proof}

Hence, for $f\in \mathcal S(\mathcal X)$ its Fourier transform $\hat f$ belongs to the space of continuous functions $h$ on $\mathcal B$ with the property that $\lim_{|x|\to\infty} |x| h(x)$ exists. It is clear that $\iota$ is an involution on this space. The proof of the next Proposition will show that it preserves the image of $\mathcal S(\mathcal X)$:

\begin{proposition}\label{Fourier-baby}
 Let $\Phi\in\mathcal S(V\times V^*)$ and let $f(\xi)=O_\xi(\Phi)$. Then:
\begin{equation}
 O_\xi(\hat\Phi) = \mathcal G(f)(\xi).
\end{equation}
\end{proposition}

\begin{proof}
We denote by $\hat\Phi^1,\hat\Phi^2$ the partial Fourier transforms with respect to the first or second argument. 
We can treat $f$ as a tempered distribution on $\mathcal B$; let $h$ be a Schwartz function on $\mathcal B$, then according to the integration formula we have:
$$\int_{\mathcal B} f(\xi) \overline{h(\xi)} d\xi = \iint \Phi(x,y) \overline{h(xy)} dx dy = $$ $$ =\iint \hat \Phi^1(x,y) \overline{\hat h\left(\frac{x}{y}\right)} |y|^{-1} dx dy = \iint \hat\Phi(x,y) \overline{\mathcal G(h)(xy)} dx dy=$$ $$ = \int_{\mathcal B} O_\xi(\hat\Phi) \overline{\mathcal G(h)}(\xi) d\xi.$$

It is easy to see that all the integrals above are absolutely convergent. Now, $\mathcal G = \mathcal F\circ \iota\circ \mathcal F$, and the operations $\mathcal F$ and $\iota$ preserve inner products, therefore:
$$\int_{\mathcal B} f(\xi) \overline{h(\xi)} d\xi = \int_{\mathcal B} \mathcal G(f)(\xi) \overline{\mathcal G(h)}(\xi) d\xi.$$

Hence, $O_\xi(\hat\Phi) = \mathcal G(f)(\xi)$.
\end{proof}

It now follows that $\mathcal G f$ not only is a function, but it belongs to $\mathcal S(\mathcal X)$. Since Fourier transform is a topological automorphism of $\mathcal S(V\times V^*)$, it follows that $\mathcal G$ is a topological automorphism of $\mathcal S(\mathcal X)$, identified with the space of coinvariants. It also follows that the image of $\mathcal S(\mathcal X)$ under Fourier transform is $\iota$-stable:

\begin{corollary} \label{corollaryFourier}
 The Fourier transform of $\mathcal S(\mathcal X)$ is the space of those (almost) smooth functions on $\mathcal B$ which in a neighborhood of infinity are equal to $|x|^{-1} h\left(\frac{1}{x}\right)$, for some $h\in \mathcal S(\mathcal B)$. Moreover, Fourier transform descends to a topological isomorphism between $\mathcal S(\mathcal X)/\mathcal S(\mathcal B)$ and the stalk\footnote{Recall that the notion of ``stalk'' used for smooth functions is the one of appendix \ref{app:cosheaves}; in particular, the germ of a smooth function at a point is determined by its derivatives.} of functions of the form: $$|x|^{-1} h\left(\frac{1}{x}\right)$$
at $\infty$ (with the obvious topology, given by the derivatives of $h$ at $0$). 
\end{corollary}

\begin{proof}
 It is clear that we have a short exact sequence:
$$ 0 \to \mathcal S(\mathcal B) \to \mathcal F\left(\mathcal S(\mathcal X)\right) \to V \to 0,$$
arising as the Fourier transform of the sequence:
$$ 0 \to \mathcal S(\mathcal B) \to \mathcal S(\mathcal X) \to \mathcal S(\mathcal X)/\mathcal S(\mathcal B) \to 0,$$
where $\mathcal S(\mathcal B)$ is endowed with its usual topology, $\mathcal F\left(\mathcal S(\mathcal X)\right)$ is endowed with the topology of $\mathcal S(\mathcal X)$ and $V$ is defined by this short exact sequence. Moreover, the first arrow is a closed embedding, and all elements of $\mathcal F\left(\mathcal S(\mathcal X)\right) $ coincide with elements of $\mathcal S(\mathcal B)$ on any compact subset of $\mathcal B$.

Since $\iota$ is a topological automorphism of $W:=\mathcal F\left(\mathcal S(\mathcal X)\right)$, this implies that $W$ consists precisely of functions as in the statement of the corollary. Such functions are sections of a Schwartz cosheaf over $\overline{\mathcal B}=\PP^1$ in an obvious way, and $\iota$ induces an isomorphism from the stalk at zero to the stalk at infinity -- the latter being equal to the quotient $V$. From this it follows that the topology on $V$ is given by the derivatives of $h$ at $0$ (where $h$ appears in the expansion of a given element at $\infty$ as in the statement).
\end{proof}

\subsection{Mellin transform} Let us now view $\mathcal B\simeq \Ga$ as a vector space, and describe the integral operator $\mathcal G$ in terms of Mellin transforms with respect to the action of $\Gm$ on $\Ga$. We normalize the action of $F^\times$ on $L^2(\mathcal B)$ as in (\ref{unitaryaction}).

By the asymptotic behavior of Lemma \ref{asympt}, any $f\in \mathcal S(\mathcal X)$ satisfies the Mellin inversion formula:
\begin{equation}\label{Mellininversion}
 f(\xi) = \int_{|\bullet|^\kappa\cdot \widehat{F^\times}} \check f(\chi) \chi(\xi) |\xi|^{-\frac{1}{2}} d\chi
\end{equation}
for every $\kappa<-\frac{1}{2}$. Here $\check f(\chi)$ denotes the Mellin transform (not to be confused with the Fourier transform $\hat f$):
\begin{equation}
 \check f(\chi) = \int_{F^\times}  |\xi|^\frac{1}{2} f(\xi) \chi^{-1}(\xi)  d^\times\xi.
\end{equation}
We do not explicate the dual measures of the formulas above and below, because we will only be interested in gamma factors, which do not depend on the measures. For this reason, we ignore the fact that our normalization of multiplicative measures is different from the one of Tate's thesis.

We claim:
\begin{lemma}
 We have $\widecheck{\mathcal G(f)} (\chi) = \gamma(\chi,\frac{1}{2},\psi)^2 \check f (\chi^{-1})$, where $\gamma(\chi,s,\psi)$ is the gamma factor of $\chi$ at $s$ (cf.\ below).
\end{lemma}

By the Mellin inversion formula (\ref{Mellininversion}), this completely characterizes the operator $\mathcal G$.

\begin{proof}
By continuity of $\mathcal G$, it is enough to prove it for $f$ in a dense subspace of $\mathcal S(\mathcal X)$, so let us assume that $f(\xi)=O_\xi(\Phi)$ with $\Phi(x,y) = \Phi_1(x) \Phi_2(y)$. Then by the integration formula we can write:
$$ \check f(\chi) = \iint \Phi_1(x) \Phi_2(y) \chi^{-1}(xy)|xy|^\frac{1}{2} d^\times x d^\times y = $$
$$ \zeta(\Phi_1,\chi^{-1},\frac{1}{2}) \cdot \zeta(\Phi_2,\chi^{-1},\frac{1}{2}),$$
where $\zeta(\Phi_i,\chi,s) = \int_{F^\times} \Phi_i(x) \chi(x) |x|^s d^\times x$ denotes the \emph{Tate integral} of $\Phi_i$ \cite{Tate}.

 Similarly, $\widecheck{\mathcal G(f)} = \zeta(\hat \Phi_1,\chi^{-1},\frac{1}{2}) \cdot \zeta(\hat \Phi_2,\chi^{-1},\frac{1}{2})$, and by the functional equation for Tate integrals we have, by definition:
$$\gamma(\chi,s,\psi)\zeta(\Phi_i,\chi,s) = \zeta(\hat\Phi_i, \chi^{-1}, 1-s).$$
This implies the claim of the lemma.
\end{proof}

\subsection{The non-split case} We discuss now the non-split version of the previous example, where $V\times V^*$ has been replaced by the space whose $F$-points are equal to the elements of a quadratic field extension $E$, under the action of the group of elements of norm $1$. In other words, we take:
$$ X = \Res_{E/F} \Ga,$$
the one-dimensional torus $T$ over $k$ defined by the short exact sequence:
$$ 1\to T \to \Res_{E/F} \Gm \xrightarrow{N^E_F} \Gm \to 1,$$
and we will be interested in the quotient stack:
$$\mathcal X = X/T.$$
The quadratic character of $F^\times$ associated to the extension $E$ will be denoted by $\eta_{E/F}$ or simply $\eta$.

Again we have a canonical isomorphism of categorical quotients: $\mathcal B:= X\sslash T \xrightarrow{\sim} \Ga$ given by the norm map.

The first thing to notice here is that the quotient stack has ``points'' corresponding to nontrivial torsors of $T$ and which are, therefore, not accounted by $F$-points of $X$; this is already evident by the fact that the map: $X\twoheadrightarrow \mathcal B$ is not surjective at the level of $F$-points. We therefore propose the following two definitions of $\mathcal M(\mathcal X)$, which can be seen to be equivalent (the first one was suggested to me by Joseph Bernstein):

\begin{enumerate}
 \item We let $T\to \GL_2$ be an embedding, and let $\mathcal \mathcal M(\mathcal X)$ denote the $\GL_2(F)$-coinvariants of $\mathcal M\left((X\times^T \GL_2)(F)\right)$ (Schwartz measures).

 \item We let $\mathcal M(\mathcal X)$ be the direct sum, over all isomorphism classes $\alpha$ of $T$-torsors, of the coinvariant spaces: $\mathcal M(X^\alpha(F))_{T^\alpha(F)}$.  

Here for a $T$-torsor $R^\alpha$ in the isomorphism class denoted by $\alpha$ we let $T^\alpha = \Aut(R^\alpha)^T$ and $X^\alpha = X\times^T R^\alpha$. (In terms of Galois cohomology, $\alpha$ can be regarded as denoting an element of $H^1(F,T)$, $T^\alpha$ is defined by its image in $H^1(F,\operatorname{Aut}(T))$ and $X^\alpha$ by its image in $H^1(F,\Aut(X))$.) Of course, in this case we have $T^\alpha \simeq T$ for all $\alpha$, but these constructions make sense in a much more general setting -- and explain inner forms appearing in the relative trace formula.

Notice that $X^\alpha$ has $F$-points if and only if $R^\alpha$ admits a $T$-equivariant morphism into $X$.
\end{enumerate}

Although the first definition is more natural and geometric, the second one is more suitable for spectral expansions, and we will be working with that.

\subsection{Integration formula, Orbital integrals}

We endow $X(F)=E$ with a Haar measure, and then the exact sequence:
$$1 \to T(F) \to E^\times \xrightarrow{N^E_F} F^\times, $$ together with the multiplicative measure $|x|^{-1} dx$ on $F^\times$ (where $dx$ is the standard additive measure on $F$ discussed in \ref{ssTamagawa}), endow $T(F)$ with a Haar measure which, by definition, satisfies the integration formula:

$$ \int_E \Phi(e) d^\times e = \int_F O_\xi(\Phi) d^\times \xi,$$
where $O_\xi(\Phi)$ is the orbital integral:

\begin{equation}
 O_\xi(\Phi) = \int_{T(F)} \Phi(\tilde \xi\cdot t) dt \,\,\text{   (here }\tilde \xi\text{: a lift of }\xi\text{ to }E\text{)}.
\end{equation}

Since, by definition, the absolute value of an element of $E$ is equal to the absolute value of its norm in $F$, we also have:
\begin{equation} \int_E \Phi(e) de = \int_F O_\xi(\Phi) d\xi.\end{equation}

The following is immediate:
\begin{lemma}\label{integration-babynonsplit}
There are compatible choices of invariant measures on $\sqcup_\alpha X^\alpha(F)$ such that for any $\Phi\in\mathcal S(\sqcup_\alpha X^\alpha)$ with image $f\in \mathcal S(\mathcal X)$ we have:
\begin{equation}
\int_{\sqcup_\alpha X^\alpha(F)} \Phi(x) dx = \int_{\mathcal B} f(\xi) d\xi.
\end{equation}
\end{lemma}
Notice that, although $X^\alpha$ is $T=T^\alpha$-isomorphic to $X$, the measures are not preserved by such a (non-canonical) isomorphism.

We define $\mathcal S(\mathcal X)$ as the cosheaf over $\mathcal B$ of functions on $\mathcal B^\reg$ obtained as orbital integrals of elements of $\mathcal S(\sqcup_\alpha X^\alpha)$. Notice that this includes the summands parametrized by all classes of $T$-torsors. Again, the standard measure on $\mathcal B=F$ gives rise to a linear map: $\mathcal M(\mathcal X)\to \mathcal S(\mathcal X)$, which can be seen (and will follow from the sequel) to be an isomorphism.

For the rest of this section we denote by $X^\alpha$ only the copy corresponding to the nontrivial torsor, and by $X$ the copy corresponding to the trivial one. The ``origins'' of $X$ and $X^\alpha$ (i.e.\ their unique $F$-points over $\xi = 0$) will be denoted by $0_X$ and $0_{X^\alpha}$. Here $T(F)=T^\alpha(F)$ is compact, therefore the asymptotic behavior of orbital integrals is easy to determine:

\begin{proposition} \label{germsnonsplit}
The restriction of $\mathcal S(\mathcal X)$ to $\mathcal B^\reg$ is equal to $\mathcal S(\mathcal B^\reg)$. In a neighborhood of zero, the sections of $\mathcal S(\mathcal X)$ are precisely those functions of the form: \begin{equation}C_1(\xi) + C_2(\xi) \eta(\xi),\end{equation} with $C_1$, $C_2$ (almost) smooth functions in one variable. Moreover, the values $C_1(0),C_2(0)$ are a basis for functionals on the fiber of $\mathcal S(\mathcal X)$ over $\xi=0$, and we have $C_1(0) =\tilde O_{0,1}(\Phi)$, $C_2(0) = \tilde O_{0,\kappa_0} (\Phi)$,
where:
\begin{equation}\label{babyipnonsplit}
\tilde O_{0,1} (\Phi) = \frac{1}{2}\Vol(T(F)) \left(\Phi(0_X) + \Phi(0_{X^\alpha})\right),
\end{equation}
\begin{equation}
\tilde O_{0,\kappa_0}(\Phi) = \frac{1}{2}\Vol(T(F)) \left(\Phi(0_X) - \Phi(0_{X^\alpha})\right).
\end{equation}
\end{proposition}

We explain the index $\kappa_0$: the Langlands dual group of $T$ is the semi-direct product of $\CC^\times$ with $\Gal(\bar F/F)$, and the $\Gal(\bar F/F)$-invariants of $\CC^\times$ consist of two elements $1,\kappa_0$. In the second (global) part of the paper we will recall that these elements parametrize characters on $H^1(F, T)$, and hence the distribution $\tilde O_{0,\kappa}$ ($\kappa = 1,\kappa_0$) is a ``$\kappa$-twisted orbital integral'' over a stable orbit.

\begin{proof}
 It is very easy to see that if $\Phi$ is supported on $X(F)$ then, close to zero, $O_\xi(\Phi)$ is equal to $0$ if $\xi$ is not a norm from $E$ and equal to a smooth function with value $\Vol(T(F)) \cdot \Phi(0_X)$ at zero if $\xi$ is a norm; similarly for the case that $\Phi$ is supported on $X^\alpha(F)$, but with $\xi$ not a norm. The result follows.
\end{proof}

\subsection{Fourier transform}

We define Fourier transform on $E$ by identifying it with its Pontryagin dual via the pairing $(x,y)\mapsto \psi(\tr(x\bar y))$, where $\psi$ is as before and $\bar y$ denotes the Galois conjugate of $y$. The chosen measure on $E$ is self-dual with respect to Fourier transform. Since we will be interested only in characters $\chi_E$ of $E^\times$ which are base change of characters of $F^\times$, i.e.\ $\chi_E = \chi\circ N_F^E $ or, equivalently, $\bar\chi_E = \chi_E$, the functional equation of Tate integrals does not change under this alternative definition of duality, i.e.:
\begin{equation}\label{localfeE}\gamma(\chi_E,s,\psi_E)\zeta_E(\Phi,\chi_E,s) = \zeta_E(\hat\Phi, \chi_E^{-1}, 1-s)
\end{equation}
for such characters $\chi_E$. The additive character $\psi_E$ will be taken to be the composition of the character $\psi$ used previously on $F$ with the trace map.

The correct way to define Fourier transform on $X^\alpha(F)=E^\alpha$ is to notice that the hermitian map: $(x,y)\mapsto \tr(x\bar y)$ extends naturally to $E^\alpha$: if we choose any isomorphism $\iota: E\to E^\alpha$ which maps $1\in E$ to the element $a\in E^\alpha$ then we have:
$$\tr(\iota x\cdot \overline{\iota y}) := \tr(N_F^E(a) x\bar y),$$
and this definition clearly does not depend on $\iota$. Then we have on $E^\alpha$, as we had on $E$:
\begin{equation}\label{Fourierontorsor-def}
 \widehat {\Phi^\alpha}(y) := \int_{E^\alpha} \Phi(x) \psi^{-1}\left(\tr(x\cdot \bar y)\right) dx,
\end{equation}
for the Haar measure defined previously, and this is self-dual, i.e.:

\begin{lemma} For $\Phi^\alpha\in C_c^\infty(E^\alpha)$,
 \begin{equation}
\widehat{\widehat{  \Phi^\alpha}}(x) = \Phi^\alpha(-x).
 \end{equation}
\end{lemma}

\begin{proof}
 In what follows, it is important to distinguish between absolute values in $F$ and in $E$, therefore we will be distinguishing them by an index. Choosing an isomorphism $\phi: E\to E^\alpha$ with $1\mapsto e$ and $N_F^E(e) = a$, and denoting the pullback of $\Phi^\alpha$ under this isomorphism by $\Phi^0$, we have:
\begin{equation} \label{Fourierontorsor} \phi^* \widehat{\Phi^\alpha}(y) = \int \Phi^\alpha(\phi x) \psi^{-1} (a x\bar y) |a|_F dx = |a|_F \widehat{\Phi^0} (ay)\end{equation}
and, similarly,
$$ \phi^* \widehat{\widehat{\Phi^\alpha}} (x) = |a|_F \widehat{\phi^* \widehat{\Phi^\alpha}} (ax) = |a|_F^2 \mathcal F\left(  \mathcal F (\Phi^0(a\cdot \bullet))\right) (ax) =$$
$$ = |a|_F^2 \cdot \frac{1}{|a|_E} \widehat{\widehat{\Phi^0}}\left(\frac{ax}{a}\right) = \widehat{\widehat{\Phi^0}} (x) = \Phi^0(-x).$$
\end{proof}

For $f\in \mathcal S(\mathcal X)$ we now define the transform $\mathcal G$ as in (\ref{defG}), except that now $\eta$ is nontrivial.

We claim:

\begin{proposition}
  Let $f\in\mathcal S(\mathcal X)$ with lift $\Phi\in \mathcal S(X)$. Then:
\begin{equation}
 O_\xi(\hat\Phi) = \mathcal G(f)(\xi).
\end{equation}
\end{proposition}

\begin{proof}
 Before we discuss the proof, let us discuss Tate integrals on $E^\alpha$: they will be defined as:
$$\zeta_E(\Phi^\alpha, \chi\circ N_F^E, s) = \int_{E^\alpha} \Phi^\alpha(x) \chi(N^E_F x) |x|^{s-1} dx,$$
where $|x|$ denotes the absolute value, extended to $E^\alpha$ via the norm map. The Tate integral is defined only for characters of the form $\chi\circ N_F^E$ in this case. If, now, $\Phi$ is a function supported on the union of $E$ and $E^\alpha$ (with corresponding restrictions denoted by $\Phi^0$ and $\Phi^\alpha$) we define:
$$\zeta_E(\Phi, \chi\circ N_F^E, s) = \zeta_E(\Phi^0, \chi\circ N_F^E, s) + \zeta_E(\Phi^\alpha, \chi\circ N_F^E, s).$$

It can be seen that, with these definitions, the local functional equation (\ref{localfeE}) still holds. 
Notice, moreover, that since $\chi_E = \chi\circ N_F^E$, we have:
\begin{equation}
 \gamma(\chi_E,s,\psi_E) = \gamma(\chi,s,\psi)\cdot \gamma(\chi\otimes \eta_{E/F},s,\psi)
\end{equation}

It is now clear, as in the split case, that the Mellin transform of $f(\xi) = O_\xi(\Phi)$, $\Phi\in\mathcal S(\mathcal X)$, can be written as:
\begin{equation}
 \check f(\chi) = \zeta_E (\Phi,\chi^{-1}\circ N_F^E, \frac{1}{2}),
\end{equation}
and similarly if $h_\xi = O_\xi(\hat \Phi)$ we have:
$$ \check h(\chi) = \zeta_E (\hat\Phi,\chi^{-1}\circ N_F^E, \frac{1}{2}).$$

Therefore:
$$ \check h(\chi) = \gamma(\chi,\frac{1}{2},\psi)\gamma(\chi\otimes \eta_{E/F},\frac{1}{2},\psi)  \check f(\chi^{-1}).$$

Both $h$ and $\mathcal G(f)$ satisfy the Mellin inversion formula \ref{Mellininversion}, since they belong to $\mathcal S(\mathcal X)$, therefore it suffices to check that their Mellin transforms coincide, which is immediate by the above calculation and an easy calculation of the Tate integrals of $G(f)$.

\end{proof}

\section{The torus quotient.} \label{sec:torus}

\subsection{} From now on $G$ will denote the group $\PGL_2$ over $F$. By $T$ we will be denoting a nontrivial torus in $G$, and $E$ will be the quadratic etale extention of $F$ such that $T = \ker(N_F^E)$. If $T$ is split we have $E=F\oplus F$, and in that case we will sometimes be identifying $T$ with some maximal torus inside of a chosen Borel subgroup, and will also be denoting it by $A$. As before, $\eta=\eta_{E/F}$ is the quadratic character associated to $E$.

The first main result of this section will be:

\begin{proposition}\label{stackisomorphism}
 Let $\mathcal Y$ denote the stack $(\Res_{E/F} G_a) / T$ with the preimage of $\xi = -1 \in \mathcal B:= (\Res_{E/F} G_a) \sslash T \simeq \Ga$ removed. Let $\mathcal Y^\times$ denote $\mathcal Y$ with the preimage of $\xi=0$ removed; then the morphism to $\mathcal B$ defines an isomorphism between $\mathcal Y^\times$ and $\mathbb A^1\smallsetminus\{-1,0\}$. 

If $\mathcal Y_1,\mathcal Y_2$ denote two distinct copies of $\mathcal Y$, the stack $\mathcal Z:= T\backslash G/T$ is isomorphic to the glueing of $\mathcal Y_1,\mathcal Y_2$ by the map:
$$ \xi\mapsto -1-\xi.$$
\end{proposition}

Among other things, this allows us to identify the GIT quotient $X\times X\sslash G$ with $\mathcal B=\Ga$. To preserve consistency of notation, the diagonal copy of $X$ in $X\times X$ will have image $-1\in \mathcal B$, and the open subset $\mathcal Y_1$ will be such that it contains the image of the diagonal copy (and hence the map $\mathcal Y_2\to \mathcal B$ will be the one defined in the previous section, while the map $\mathcal Y_1\to \mathcal B$ will be  obtained from that by $\xi\mapsto -1-\xi$). We let $\mathcal B^\reg=\mathcal B_{\mathcal Z}^\reg$ be the complement of $\{0,-1\}$ in $\mathcal B$.  

\subsection{The open subset in the split case}

We consider first the case $E=F\oplus F$, let $B^+, B^-$ denote the two Borel subgroups which contain $T=A$ and let $\mathcal Z_1:= A\backslash (B^- B^+ \cap B^+ B^-) / A$, which is open in $\mathcal Z$. 

\begin{lemma}\label{opensetisom}
\begin{enumerate}
              \item 
 The map:\footnote{For notational clarity, we formulate in terms of $F$-points some statements which should, strictly speaking, be formulated in terms of schemes.}
\begin{equation}\label{EtoMat}
 F\oplus F \ni (x,y)\overset{\iota}{\mapsto} \left(\begin{array}{cc}
                                   1& x\\ & 1
                                  \end{array}\right) \left(\begin{array}{cc}
                                                            1\\ y& 1
                                                           \end{array}\right) = \left(\begin{array}{cc}
                                   1& x\\ y & 1+ xy 
                                  \end{array}\right)  \in G
\end{equation}
descends to an isomorphism: 
\begin{equation}
 \mathcal Y \xrightarrow{\sim} \mathcal Z_1.
\end{equation}

\item Let $w$ be an element in the $F$-points of the non-identity component of the normalizer of $A$, then the automorphism $g\mapsto {^wg}$ fixes the preimage of $\mathcal Z_1$ in $G$ and has the property that:
\begin{equation}
 ^w \iota(x,y) \sim \iota \left(y(1+xy), \frac{x}{1+xy}\right)
\end{equation}
modulo the action of $A\times A$.
             \end{enumerate}
\end{lemma}

\begin{proof}
 Direct calculation. For the first statement, one easily sees that $A\backslash (B^- B^+ \cap B^+ B^-)$ is isomorphic to the variety of matrices of the form (\ref{EtoMat}) with $xy\ne -1$, and that, thinking of those matrices this way, the map $\iota$ is $A$-equivariant. The second statement is immediate.
\end{proof}

\subsection{} Let $X=A\backslash G$ and let $w$ denote the nontrivial $G$-automorphism of $X$. We have a natural isomorphism of stacks:
$$X\times X/G \ni (x_1,x_2)\mapsto x_1x_2^{-1} \in A\backslash G/A,$$
and applying the ``$w$''-automorphism on the second copy of $X$ induces an isomorphism of open substacks:
$$\mathcal Z_2\xrightarrow{\sim} \mathcal Z_1,$$
where $\mathcal Z_2 = A\backslash (BwB\cap B^-wB^-)/A$.

Combined with the isomorphism of Lemma \ref{opensetisom}, this proves Proposition \ref{stackisomorphism} in the split case.

\subsection{The non-split case}

In the above setting, but with $E$ now denoting a field extension, the cocycle which takes the nontrivial element $\sigma$ of $\Gal(E/F)$ to the inner automorphism of $G$ by $w$ (viewed also as an automorphism of $A$) defines a form of both $\PGL_2$ and $A$ over $F$. The form of $A$ is $T=\ker(N_F^E)$, while the form of $\PGL_2$ could be split or non-split according as the cocycle chosen lifts to $\GL_2$ or not. (This depends on the representative $w\in \mathcal N(A)(F)$ chosen, more precisely on whether the negative of the quotient of its eigenvalues is a norm from $E$ or not.) This shows, in particular, that:
\begin{equation}
 T\backslash G/T \simeq T\backslash G'/T,
\end{equation}
 where $G=\PGL_2$ and $G'$ is an inner form of $G$ which splits over $E$.

Notice that $w$ preserves the open substacks $\mathcal Z_1,\mathcal Z_2$ of $\mathcal Z$ and hence defines forms of those.

At the same time, we have seen that the ``$w$'' automorphism on $A\backslash G$ corresponds, under the map $\iota$ of Lemma \ref{opensetisom}, to the automorphism:
$$ \tau: (x,y)\mapsto \left(y(1+xy), \frac{x}{1+xy}\right)$$
of the subset of $(x,y)\in k\oplus k$ with $xy\ne -1$. This is the same as the composition of the automorphism: $(x,y)\mapsto (y,x)$ with the action of $(1+xy, \frac{1}{1+xy})\in A$, and therefore the form of the quotient stack defined by the cocycle $\sigma\mapsto \tau$ is isomorphic to:
$$\mathcal Y_1 \simeq  \left(\Res_F^E (\Ga)\smallsetminus (N_F^E)^{-1}(-1)\right) /T.$$

Therefore, even in the non-split case we have: $\mathcal Z_1\simeq \mathcal Y_1$ and, similarly, $\mathcal Z_2\simeq\mathcal Y_2$, which completes the proof of Proposition \ref{stackisomorphism}.

\subsection{Schwartz functions and orbital integrals}

We define $\mathcal M(\mathcal Z)$ to be the $G$-coinvariant space of $\mathcal M(X\times X)$ (the Fr\'echet space of Schwartz measures) in the split case. In the non-split case we must, as before, use one of the following equivalent definitions:
\begin{enumerate}
 \item For some embedding of $G$ into $\GL_n$ we let $\mathcal M(\mathcal Z)$ be the space of $\GL_n(F)$-coinvariants of $\mathcal M\left(((X\times X)\times^G \GL_n)(F)\right)$.

 \item We let $\mathcal M(\mathcal Z)$ be the direct sum, over all isomorphism classes of $T$-torsors $R^\alpha$, of the coinvariant space: $\mathcal M((X^\alpha\times X^\alpha)(F))_{G^\alpha(F)}$, where $G^\alpha = \Aut_G(R^\alpha\times^T G)$, $X^\alpha = T\backslash G^\alpha$.  Equivalently, $G^\alpha$ ranges over inner forms of $G$ which split over $E$. 
\end{enumerate}

By the above isomorphisms of stacks, we also have: 
\begin{equation}\label{Schwartzisom}\mathcal M(\mathcal Z) = \left(\mathcal M(\mathcal Y_1) \oplus \mathcal M(\mathcal Y_2)\right) / \mathcal M(\mathcal Y_1\cap \mathcal Y_2).
\end{equation}

We let $\mathcal S(\mathcal Z)$ denote the cosheaf on $\mathcal B$ of functions on $\mathcal B^\reg$ which are obtained as regular orbital integrals of Schwartz functions on $\sqcup_\alpha (X^\alpha\times X^\alpha)$.
\begin{equation}
\xi\mapsto O_\xi(\Phi) = \int_{G(F)} \Phi(\tilde\xi\cdot g),
 \end{equation}
where $\tilde \xi$ is a representative for the orbit parametrized by $\xi$.

The results of section \ref{sec:babycase} immediately imply:
\begin{proposition}
The restriction of $\mathcal S(\mathcal Z)$ to $\mathcal B^\reg = \Ga\smallsetminus\{0,-1\}$ is equal to the cosheaf of Schwartz functions $\mathcal S(\mathcal B^\reg)$. In neighborhoods of $\xi=0$ and $\xi=-1$ they have the behavior of the germs of Propositions \ref{germs-split}, \ref{germsnonsplit} around zero.
\end{proposition}

The choice of Haar measure on $G(F)$ does not matter for the definition of the sheaf $\mathcal S(\mathcal Z)$, and again by a ``lift of an element of $\mathcal S(\mathcal Z)$ to $\mathcal S(X\times X)$'' we will implicitly mean an element of $\mathcal S(X\times X)$ together with a choice of Haar measure on $G(F)$. However, we would now like to define a linear isomorphism:
\begin{equation}\label{measurestoorbital}
 \mathcal M(\mathcal Z)\xrightarrow{\sim} \mathcal S(\mathcal Z).
\end{equation}
We do this locally on the open cover $\mathcal Z_1\cup \mathcal Z_2$ by using the identification of $\mathcal M(\mathcal Z_i)$ with $\mathcal M(\mathcal Y_i)$, together with the identification (\ref{measurestoorbital}) from the previous section: $\mathcal M(\mathcal Y_i) \simeq \mathcal S(\mathcal Y_i)$, which gives rise to an map to $\mathcal S(\mathcal Z)$:
$$ \mathcal M(\mathcal Z_i) \xrightarrow{\sim} \mathcal M(\mathcal Y_i) \xrightarrow{\sim} \mathcal S(\mathcal Y_i) \to \mathcal S(\mathcal Z).$$

Equivalently, this isomorphism arises from the standard additive measure on $\mathcal B=F$. Notice that the same integration formula as in the previous section (Propositions \ref{integration-babysplit} and \ref{integration-babynonsplit}) follows from the local isomorphisms of stacks:

\begin{lemma}\label{integration-torus}
There are compatible choices of invariant measures on $\mathcal B$ (additive), $G$ and $\sqcup_\alpha (X^\alpha\times X^\alpha)$ such that for any $\Phi\in\mathcal S(\sqcup_\alpha (X^\alpha\times X^\alpha))$ with image $f\in \mathcal S(\mathcal Z)$ we have:
\begin{equation}
\int_{\sqcup_\alpha (X^\alpha\times X^\alpha)} \Phi(x_1,x_2) dx_1 dx_2 = \int_{\mathcal B} f(\xi) d\xi.
\end{equation}
\end{lemma}

\subsection{Inner products} \label{ssiptorus}

In order not to introduce excessive notation, we will not reserve any symbols for the irregular distributions on $\mathcal Z$ which are the analogs of $\tilde O_0,\tilde O_u,\tilde O_{0,1}, \tilde O_{0,\kappa_0}$ of \S \ref{sec:babycase}, with one exception:

Let $\alpha$ be a class of $T$-torsors. For $\Phi_1,\Phi_2 \in \mathcal S(X^\alpha)$, and an invariant measure $dx$ on $X^\alpha$ (where ``invariant'' means, of course, invariant under the pertinent inner form of $G$ on each copy) we define the \emph{inner product} of $\Phi_1,\Phi_2$ as:
$$\left<\Phi_1,\Phi_2\right> = \int_{X^\alpha} \Phi_1(x)\Phi_2(x) dx,$$
i.e.\ as a bilinear form. Clearly, this extends continuously to $\mathcal S(X^\alpha\times X^\alpha)$, and for an element $\Phi$ of the latter we will simply write $\left<\Phi\right>$. 

Now, given $f\in \mathcal S(\mathcal Z)$, choose a pair $(\Phi,(dg_\alpha)_\alpha)$ consisting of an element $\Phi = \sum \Phi^\alpha \in \oplus_\alpha \mathcal S(X^\alpha\times X^\alpha)$ and a collection of Haar measures on the inner forms $G^\alpha$ such that $f$ arises as the regular orbital integrals of $\Phi$ with respect to those measures. Let:
\begin{equation}
 (-1)^\alpha = \begin{cases} 1, & \mbox{ if $\alpha$ corresponds to the trivial torsor,}\\
  -1 & \mbox{otherwise;}
 \end{cases}
\end{equation}
that is, we are identifying $H^1(F,T)$ with $\Z/2$ in the nonsplit case.

Then we define the \emph{``inner product''} of $f$ as:
\begin{equation}\label{eqiptorus}\left<f\right>:= (F^\times:N_F^EE^\times)^{-1}\cdot \sum_\alpha  (-1)^\alpha \Vol(T(F)_0)\left<\Phi^\alpha\right>,
\end{equation}
where we have implicitly chosen a decomposition of $dg$ as an invariant measure on $T(F)$ times a measure on $T\backslash G^\alpha(F)$ in order to define both the inner product on $X^\alpha$ and the volume of $T(F)_0$ according to the recipe of \S \ref{ssbabyorbital} (of course, in the non-split case $T(F)_0=T(F)$ so no recipe is needed). Clearly, the definition does not depend on this choice, so we have a well-defined functional on $\mathcal S(\mathcal Z)$. The following is easy to see by the results of the previous section:

\begin{lemma}
 In the split case $\left<f\right>$ is equal to the distribution $\tilde O_0(f)$ of \eqref{babyipsplit} when a neighborhood of $\xi=-1$ of $\mathcal Z$ is identified with a neighborhood of $\xi=0$ of $\mathcal X$ according to Proposition \ref{stackisomorphism}. In the nonsplit case, $\left< f\right>$ is equal to the distribution $\tilde O_{0,\kappa_0}(f)$ of \eqref{babyipnonsplit} under the same identification.
\end{lemma}

\section{The Kuznetsov quotient with nonstandard functions.} \label{sec:Kuznetsov}

\subsection{} The Kuznetsov trace formula is the relative trace formula for $\mathcal S(X,\mathcal L_\psi)\otimes\mathcal S(X,\mathcal L_{\psi}^{-1})$, where $X$ is the quotient of $\PGL_2$ by a nontrivial unipotent subgroup $N$ and $\mathcal L_\psi$ is the complex $G$-line bundle on $X$ defined by a character $\psi$ of $N(F)$. Here, however, we will extend it to nonstandard sections of this line bundle, that is, sections which are not Schwartz, with prescribed asymptotic behavior at infinity. One can identify $X$ with the quotient by $\{\pm 1\}$ of two-dimensional affine space, minus the origin, and ``infinity'' is precisely the partial compactification of this space by $\mathbb P^1$ at ``infinity''.

Let us start in a slightly different way: Let $G = \Aut(\mathbb P^1)$ and let $\bar X$ denote the total space of the line bundle $O(-2)$ over $\mathbb P^1$; it is $G$-linearizable, i.e.\ it carries an action of $G$ which commutes with the natural action of $\Gm$. We denote by $X$ the complement of the zero section -- it is homogeneous under $G$, and stabilizers are unipotent subgroups. There is a unique, up to the action of $\Gm(F)$, $G(F)$-linear complex line bundle on the $F$-points of $\bar X$ on which the stabilizers of points on $X$ act by a nontrivial unitary character; we fix such a line bundle, and denote it by $\mathcal L_\psi$. Over $\mathbb P^1$ this is $G$-isomorphic (non-canonically) to the trivial line bundle. 
For any subset $S$ of $\bar X \times \bar X$ we will denote by $S^+$ the open subset of $S$ lying over the open $G$-orbit on $\mathbb P^1\times\mathbb P^1$. If $S$ is stable under the diagonal action of $G$, then so is $S^+$.

\subsection{Orbits} \label{ssorbits}

Let $N$ be the subgroup of upper triangular unipotent matrices in $\PGL_2$, $N^-$ the subgroup of lower triangular matrices, both identified with the additive group $\Ga$ in the usual way, and $A$ the torus of diagonal elements. Fix a nontrivial unitary character $\psi$ of $F$, the same character that we used for Fourier transforms in previous sections. We claim that there is a canonical map from $(X\times X)^+$ to the open subset:
$$ \Gm \simeq \left\{ \left.N\left(\begin{array}{cc} \xi \\ & 1\end{array}\right) N^- \right| \xi\in \Gm\right\}$$
of $N\backslash G/N^-$. 

Indeed, let $(x,y)\in (X\times X)^+$, then there is a unique isomorphism of the triple $(G,G_x,G_y)$ with $(\PGL_2,N, N^-)$ such that $G_x$ acts on the fiber of $\mathcal L_\psi$ by the character $\psi$ of $N(F)=\Ga(F)$ (standard isomorphism), and a unique isomorphism such that $G_y$ acts on the fiber of $\mathcal L_\psi^{-1}$ by the character $\psi^{-1}$ of $N^-(F)=\Ga(F)$. Then $\left(\begin{array}{cc} \xi \\ & 1\end{array}\right)$ is the unique element of $A$ which conjugates one isomorphism to the other; our convention to distinguish between $\xi$ and $\xi^{-1}$ is that as $(x,y)$ approach the complement of $(X\times X)^+$ in $X\times X$, $\xi$ 
goes to infinity. This defines an isomorphism of quotient stacks (varieties):
\begin{equation}\label{Kuzquotientopen}(X\times X)^+/G \to \Gm.
\end{equation}
We embed $\Gm\hookrightarrow \mathcal B:=\Ga \hookrightarrow \overline{\mathcal B}:=\PP^1$, and set $\mathcal B^\reg=\mathcal B_{\mathcal W}^\reg = \mathcal B\smallsetminus\{0\}$. We will sometimes be denoting the regular set by $\mathcal B^\times$.

\begin{lemma}\label{rationalmap}
 The map (\ref{Kuzquotientopen}) extends to a rational map: 
\begin{equation}\label{rationalmapeq}
\bar X\times X \to \overline{\mathcal B},
\end{equation}
which is regular away from the complement of $(\PP^1\times X)^+$ in $\PP^1\times X$, that is: away from the set of points $(p,x)\in \PP^1\times X$ such that $x$ lies in the fiber over $p$.
\end{lemma}

The reader should keep in mind that \emph{$\xi =0$ corresponds to points on $(\PP^1\times X)^+$, while $\xi =\infty$ corresponds to the complement of $(X\times X)^+$ in $X\times X$.} This paradoxical way of parametrizing orbits plays a role when discussing Fourier transforms (where the vector space structure imposed on $\mathcal B$ is important).

There is a certain degree of arbitrariness in choosing the ``standard'' identifications of $N,N^-$ with $\Ga$; therefore, there is nothing special about the orbit above $\xi=1$. However, this choice should be compared to the choice that the ``irregular'' orbits of $\mathcal Z$ map to $\{0,-1\}\in\mathcal B$ in the discussion of section \ref{sec:torus}; both are essentially choices of a generator of the ring of invariants, they are related and should be changed simultaneously.

\subsection{Orbital integrals} \label{ssorbital-Kuz}

In what follows, we try to make explicit the choices made in order to think of orbital integrals on the Kuznetsov trace formula as functions on $\mathcal B^\reg\simeq F^\times$. The reader may wish to skip straight to (\ref{Kuzorbital}).

Let $\Phi_1, \Phi_2$ be smooth sections of $\mathcal L_\psi$, resp.\ $\mathcal L_\psi^{-1}$, of compact support on $X$. Fix a Haar measure on $G=G(F)$. At a first stage, we define the (regular) orbital integrals of $\Phi_1\otimes \Phi_2$ to be the $G$-invariant section of $\mathcal L_\psi\boxtimes \mathcal L_\psi^{-1}$ on $(X\times X)^+$ obtained by integrating $\Phi_1\cdot \Phi_2$, that is:
\begin{equation}
 O_{(x,y)} (\Phi_1\otimes\Phi_2) = \int_G g\cdot (\Phi_1,\Phi_2) (x,y) dg,
\end{equation}
where $g\cdot$ denotes the right regular representation under the diagonal action of $G$. 

At a second stage, we would like to represent these orbital integrals as functions on $\mathcal B^\times :=\mathcal B\smallsetminus\{0\} = F^\times$. Let $(x_0,y_0)\in (X\times X)^+$, with image $1\in \mathcal B$, let $A$ denote the unique torus in $G$ which normalizes both $G_{x_0}$ and $G_{y_0}$, identified with $\Gm$ according to the image of $(x_0a,y_0)$ in $\mathcal B$ ($a\in A$). For $\xi\in \Gm$ corresponding to $a\in A$, let:
\begin{equation}
 O_\xi (\Phi_1\otimes\Phi_2) = O_{(x_0,y_0)} (a\cdot \Phi_1\otimes\Phi_2),
\end{equation}
where $a$ acts on the section $O_{(x_0,y_0)} (\Phi_1\otimes\Phi_2)$ by the regular representation on the first coordinate; hence, $O_\xi (\Phi_1\otimes\Phi_2)$ is understood as an element in the fiber of $\mathcal L_\psi$ over $(x_0,y_0)$. We remark that $a\cdot $ denotes the action of $a$ as an element of $G$; the natural action of $\Gm$ on $O(-2)$ by dilations does not extend to the line bundle $\mathcal L_\psi$. When, later, we will replace $\mathcal L_\psi$ by the trivial line bundle, we will be denoting the action of $\Gm$ by dilations by $\mathscr L_a$ (where $\mathscr L$ is supposed to be reminiscent of ``left action'' in terms of the torus acting on $X\simeq N\backslash G$), in order to avoid confusion. Notice that the orbit map: $\mathcal B^\times \ni a \mapsto (x\cdot a, y) \in (X\times X)^+$ extends to:
\begin{equation}\label{BinX}
 \mathcal B \to (\bar X\times X)^+.
\end{equation}

Finally, we choose an isomorphism of the fiber of $\mathcal L_\psi\boxtimes \mathcal L_\psi^{-1}$ over $(x_0,y_0)$ with $\CC$, in order to consider $O_\xi (\Phi_1\otimes\Phi_2)$ as a complex-valued function on $F^\times$, as $\xi$ varies. This last choice of isomorphism only affects the orbital integrals by a common scalar multiple, and will be reflected in our choice of unramified sections in the fundamental lemma (e.g. we will ask that some sections be equal to ``1'' at $(x_0,y_0)$, which only makes sense after choosing this isomorphism).

Explicitly, if we identify the stabilizers of $x_0$, $y_0$ with $N, N^-$ such that they act on $\mathcal L_\psi$, resp.\ $\mathcal L_\psi^{-1}$ by $\psi,\psi^{-1}$, respectively, and if we trivialize the fiber in order to think of $\Phi_1\otimes\Phi_2$ as an element of $C^\infty(N\backslash G\times N^-\backslash G,\psi\otimes\psi^{-1})$, then:
\begin{equation}\label{Kuzorbital}O_\xi(\Phi_1\otimes\Phi_2) = \int_G \Phi_1(\left(\begin{array}{cc} \xi & \\ & 1\end{array}\right)g) \Phi_2(g) dg.
\end{equation}

Given these choices -- that is, the embedding (\ref{BinX}) and the identification of the fiber over $1$ with $\CC$, we can easily see:
\begin{lemma} \label{Kuztrivialization}
\begin{enumerate}
               \item The action map: \begin{equation}
                                      \mathcal B \times G \to (\bar X \times X)^+
                                     \end{equation}
is an isomorphism.
\item The chosen trivialization of the fiber over $(x_0,y_0)$, the action of $\mathcal B^\times \times G$ (here $\mathcal B^\times$ is acting as the torus $A$ on the first copy, as before), and the above isomorphism give rise to a $G$-equivariant isomorphism of:
$$\left.\mathcal L_\psi \boxtimes \mathcal L_\psi^{-1}\right|_{(\bar X \times X)^+}$$
with the trivial line bundle over $\mathcal B \times G $.
              \end{enumerate}
\end{lemma}

\subsection{Integration formula}

Identifying sections $\Phi\in \mathcal S\left(\left.\mathcal L_\psi \boxtimes \mathcal L_\psi^{-1}\right|_{(\bar X \times X)^+}\right)$ with sections of the trivial line bundle according to Lemma \ref{Kuztrivialization}, we have:

\begin{lemma}\label{Kuzintegration}
 For suitable choices of a $G$-invariant measure on $X=X(F)$ and a Haar measure on $G=G(F)$, we have:
\begin{equation}
 \int_{(\bar X \times X)^+} \Phi (x,y) d(x,y) = \int_{\mathcal B} O_\xi(\Phi) |\xi|^{-2} d\xi,
\end{equation}
where $d\xi$ is our fixed, standard additive measure on $\mathcal B \simeq F$.
\end{lemma}

\subsection{Nonstandard sections} \label{ssnonstandard}

We let $\mathcal M(\bar X\times X, \mathcal L_\psi\boxtimes \mathcal L_\psi^{-1})$ (resp.\ $\mathcal S(\bar X\times X, \mathcal L_\psi\boxtimes \mathcal L_\psi^{-1})$) denote the Schwartz cosheaf over $\bar X\times X$ consisting of smooth measures (resp.\ functions) on $X\times X$, valued in $\mathcal L_\psi\boxtimes \mathcal L_\psi^{-1}$, with the following properties:
\begin{itemize}
 \item the restriction of the cosheaf to $X\times X$ coincides with the standard cosheaf of Schwartz measures (resp.\ functions) valued in $\mathcal L_\psi\boxtimes \mathcal L_\psi^{-1}$; 
 \item in a neighborhood of $\mathbb P^1\times X$ they are finite sums of the form $\sum_i f_i F_i$, where:
\begin{enumerate}
 \item 
 the $f_i$'s are $\mathcal L_\psi\boxtimes\mathcal L_\psi^{-1}$-valued Schwartz functions on $\bar X\times X$;
 \item the $F_i$ are scalar-valued measures (resp.\ functions) on $X\times X$ which are $G$-invariant in the second coordinate\footnote{This is just one of many equivalent ways of describing our measures, and the reader should not get confused trying to figure out the purpose of invariance in the second coordinate; the point is that our resulting measures will be of Schwartz type in the second coordinate, which is taken care of by the $f_i$'s, so we only need the $F_i$'s in order to describe their asymptotic behavior in the first coordinate.}, and in the first coordinate are annihilated asymptotically by the operator:
\begin{equation}\label{asymptotics}
 \left(1- \delta^{-\frac{1}{2}} (a) \mathscr L_a\right)\cdot \left(1- \eta_{E/F}\delta^{-\frac{1}{2}} (a)\mathscr L_a\right).
\end{equation}
\end{enumerate}
\end{itemize}

Our notation is hiding the details of the asymptotics, and is just replacing $X$ by $\bar X$ to remind that these measures are not Schwartz on $X$; however, they cannot be considered as smooth measures on $\bar X$. 

We explain what it means, for a \emph{scalar-valued} function or measure on $X$, to be asymptotically annihilated by (\ref{asymptotics}). 
Thinking of $\Gm$ as the dilation group of the bundle $O(-2)$ over $\mathbb P^1$, we have an $L^2$-isometric action of it on functions on $X$ given by:
\begin{equation}
 \mathscr L_a f(x) = \delta(a)^{-\frac{1}{2}} f(ax),
\end{equation}
where $\delta(a)$ is the inverse of the character by which the non-normalized action of $\Gm$ transforms an invariant measure on $X$, written suggestively so that in an identification of $X$ with $N\backslash \PGL_2$ it corresponds to: $$\left(\begin{array}{cc}
                        a \\ & 1                                                                                                                                                                                                   \end{array}\right)\mapsto |a|.$$

Similarly, we have an $L^2$-isometric action of $\Gm$ on measures on $X$ given by:
\begin{equation}
 \mathscr L_a \mu(x) = \delta(a)^{\frac{1}{2}} \mu(ax),
\end{equation}
and of course the map from functions to measures: $f \mapsto f dx$ is equivariant with respect to these actions.

Hence, ``asymptotically annihilated'' by (\ref{asymptotics}) means that applying the operator (\ref{asymptotics}) to the given function/measure produces a function/measure which is supported away from $\mathbb P^1$. Thus, these functions can be identified in a neighborhood of $\mathbb P^1$ with elements of a representation $\pi$ of the form:
\begin{equation}\label{asymptoticrep}
 0\to I(\delta^{\frac{1}{2}}) \to \pi \to I(\eta_{E/F}\delta^\frac{1}{2})\to 0,
\end{equation}
where the sequence is nonsplit if (and only if) $\eta_{E/F}$ is trivial. Here $I(\bullet)$ denotes the principal series representation obtained by normalized induction from the character $\bullet$. However, the isomorphisms with principal series are not canonical, and we prefer to think of $\pi$ as the space of smooth functions on $X$ which are annihilated by (\ref{asymptotics}).

\subsection{Coinvariants}

We let $\mathcal M(\mathcal W)$ denote the $G$-coinvariants of $\mathcal M(\bar X\times X, \mathcal L_\psi\boxtimes \mathcal L_\psi^{-1})$. Here the letter $\mathcal W$ is reminiscent of a stack, but for us it is just formal notation, because I do not know how to make sense of $\mathcal L_\psi\boxtimes\mathcal L_\psi^{-1}$ as a bundle on the stack. We denote by $\mathcal M(\mathcal W^+)$ the $G$-coinvariants of those measures which are almost supported on $(\bar X\times X)^+\subset \bar X\times X$.

Using the trivializations of Lemma \ref{Kuztrivialization}, we have a map from $\mathcal M(\mathcal W)$ to smooth measures on $\mathcal B^\times$; as we shall see, this map is injective, so we feel free not to distinguish between an element of $\mathcal M(\mathcal W)$ and the corresponding measure on $\mathcal B^\times$. We let $\mathcal S(\mathcal W)$ be the space of functions on $\mathcal B^\times$ which are obtained as ``regular'' orbital integrals (understood as in \S \ref{ssorbital-Kuz}), with respect to a Haar measure on $G$, of elements of $\mathcal S(\bar X\times X, \mathcal L_\psi\boxtimes \mathcal L_\psi^{-1})$. According to Lemma \ref{Kuzintegration}, the map $\mu\mapsto \frac{\mu}{|\xi|^{-2} d\xi}$ is a linear isomorphism:

\begin{equation}\label{Kuzmeastoorb}
 \mathcal M(\mathcal W)\to \mathcal S(\mathcal W).
\end{equation}

We would like to understand the spaces $\mathcal M(\mathcal W^+)$, $\mathcal M(\mathcal W)$ as sections over $\mathcal B$, $\overline{\mathcal B}$ of the $G$-coinvariants of the push-forward to $\overline{\mathcal B}$ of the Schwartz cosheaf $\mathcal M(\bar X\times X, \mathcal L_\psi\boxtimes \mathcal L_\psi^{-1})$, in order to take advantage of the results of appendix \ref{app:cosheaves}. This is not directly possible, as the map \eqref{rationalmapeq} is rational, not regular. We will see, however, in Proposition \ref{propnoirregular} that the stalk over the irregular locus $S$ of \eqref{rationalmapeq} does not contribute at all to $\mathcal M(\mathcal W)$; thus, we may indeed view $\mathcal M(\mathcal W)$ as (sections of) the flabby cosheaf of coinvariants of the push-forward of $\mathcal M((\bar X\times X)\smallsetminus S, \mathcal L_\psi\boxtimes \mathcal L_\psi^{-1})$.

\subsection{Limiting behavior at 0}

Let $\mathcal X$ be the quotient stack of section \ref{sec:babycase}, that is: $\mathcal X = \Res_{E/F} \Ga/ T$, where $T = \ker N_F^E$. We consider elements of $\mathcal M(\mathcal X)$ as measures on $\mathcal B\smallsetminus \{0\}=F^\times$, as we do with elements of $\mathcal M(\mathcal W)$. Recall that $\mathcal M(\mathcal W^+)$ denotes the sections of this cosheaf over $(\bar X\times X)^+$, and $\mathcal S(\mathcal W^+)$ their images under the operation of orbital integrals.

\begin{proposition} \label{propKuzstalk}
As spaces of measures on $\mathcal F^\times$, we have:
\begin{equation}
 \mathcal M(\mathcal W^+) = |\bullet|^{-1}\mathcal M(\mathcal X). 
\end{equation}
\end{proposition}
The symbol $|\bullet|^{-1}$ denotes multiplication of a measure $\mu(\xi)$ by $|\xi|^{-1}$.

\begin{proof}
 Using the isomorphisms of Lemma \ref{Kuztrivialization}, the space $\mathcal M((\bar X\times X)^+, \mathcal L_\psi\boxtimes \mathcal L_\psi^{-1})$ can be identified with a space of scalar-valued measures on $\mathcal B^\times \times G$ with the following properties:
\begin{itemize}
 \item away from a neighborhood of $\{0\}\times G$ they coincide with Schwartz measures;
 \item in a neighborhood of $\{0\}\times G$ they are equal to a Schwartz function on $\mathcal B\times G$ times a measure $\mu(b) dg$ with $dg$ a Haar measure on $G$ and:
\begin{equation}\label{muproperty} \mu(b) - \mu(ab) - \eta_{E/F}(a) \mu(ab) + \eta_{E/F}(a)\mu(a^2b) = 0,\end{equation}
for every $a\in F^\times$.
\end{itemize}
Their $G$-coinvariants coincide with their push-forwards to $\mathcal B^\times$, which are characterized by the analogous properties (i.e.\ Schwartz away from $0$ and same condition on the measure $\mu$). By the explicit description of $\mathcal M(\mathcal X)$ in Propositions \ref{germs-split} and \ref{germsnonsplit}, the claim follows.
\end{proof}

\begin{corollary} \label{XinW} 
 As spaces of functions on $\mathcal B^\times$, we have:
\begin{equation} \label{isomW-X}
 \mathcal S(\mathcal W^+) = |\bullet|\mathcal S(\mathcal X). 
\end{equation}
\end{corollary}

This follows immediately from the integration formula of Lemma \ref{Kuzintegration}. In the next subsection we will identify the limiting behavior of an element of $\mathcal S(\mathcal W^+)$ as $\xi\to 0$ in terms of invariant distributions supported on the fiber over $\xi=0$.

\subsection{Explication}

We would now like to explicate the ``irregular'' distributions that determine the limiting behavior at zero.

Let $V$ be defined by the short exact sequence of Fr\'echet spaces:
$$0\to \mathcal S(X\times X, \mathcal L_\psi\boxtimes \mathcal L_\psi^{-1})\to \mathcal S(\bar X\times X, \mathcal L_\psi\boxtimes \mathcal L_\psi^{-1}) \to V \to 0.$$
(Recall that $\mathcal S(\bar X\times X, \mathcal L_\psi\boxtimes \mathcal L_\psi^{-1}) $ denotes the nonstandard sections defined in \ref{ssnonstandard} and not really sections of the line bundle over $\bar X\times X$.) That is, in the language of appendix \ref{app:cosheaves}, $V$ is the stalk over $\PP^1\times X$ of the Schwartz cosheaf whose global sections are $\mathcal S(\bar X\times X, \mathcal L_\psi\boxtimes \mathcal L_\psi^{-1}) $.

Let $\mathscr G_\psi$ denote the Fr\'echet space of germs of smooth sections of $\mathcal L_\psi$ around $\PP^1$. Since $\mathcal L_\psi$ is trivializable over $\PP^1$, this space is isomorphic (non-canonically) to the space of germs of smooth functions around $\PP^1$. Then we have an isomorphism of $G\times G$-representations:
\begin{equation}\label{germatzero}  V\simeq (\mathscr G_\psi\widehat\otimes_{\mathcal C^\infty(\PP^1)} \pi) \widehat\otimes \mathcal S(X),\end{equation}
(completed, projective tensor products), where $\pi$ is as in (\ref{asymptoticrep}).

Clearly, the \emph{fiber} over $\PP^1\times X$ is obtained by evaluation of the element of $\mathscr G_\psi$ on $\PP^1$, which gives a map:
\begin{equation}\label{german}\mathcal S(\bar X\times X, \mathcal L_\psi\boxtimes \mathcal L_\psi^{-1}) \to \left( C^\infty(\PP^1,\mathcal L_\psi) \otimes_{\mathcal C^\infty(\PP^1)} \pi\right) \otimes \mathcal S(X,\mathcal L_\psi^{-1}).
\end{equation}
The first tensor product is isomorphic to $\pi$, depending on a trivialization of $\mathcal L_\psi$ over $\PP^1$; for simplicity, but to remember that this isomorphism is not canonical, we will be denoting it by $\pi'$. Since the elements of $\pi$ can be thought of as sections of the trivial line bundle on $X$ annihilated by the operator (\ref{asymptotics}), elements of $\pi'$ should be thought of as similar sections of the pullback of $\mathcal L_\psi$ under the projection map: $X\to \PP^1$. 
We are going to encode the $G$-coinvariants of the fiber in two ``irregular'' orbital integrals.

For now we will define those irregular orbital integrals only for elements of $\mathcal S(\mathcal W^+)$.  Let $f\in \mathcal S(\mathcal W^+)$ and $(\Phi,dg)$ be a pair consisting of an element $\Phi\in \mathcal S((\bar X\times X)^+, \mathcal L_\psi\boxtimes \mathcal L_\psi^{-1})$ and a Haar measure on $G$ such that $f$ is obtained by the orbital integrals of $\Phi$. Under the isomorphisms of Lemma \ref{Kuztrivialization} (see also the proof of Proposition \ref{propKuzstalk}), $\Phi$ can be written as a function on $\mathcal B\times G$ of the form:
\begin{eqnarray*}
|\xi|\left(-h_1(\xi,g)\ln|\xi| + h_2(\xi,g)\right)& \mbox{ in the split case } (\eta=1), \\
|\xi|\left(h_1(\xi,g) + \eta(\xi) h_2(\xi,g)\right)& \mbox{ in the non-split case } (\eta\ne 1)
\end{eqnarray*}
where $h_i(\xi,g)$ are Schwartz functions on $\mathcal B\times G$. 

In the split case we set:
\begin{eqnarray}
 \tilde O_{0,\delta^\frac{1}{2}}(f) = \int_G h_1 (0,g) dg,\\
 \tilde O_{u, \delta^\frac{1}{2}}(f) = \int_G h_2(0,g) dg.
\end{eqnarray}

In the nonsplit case we set:
\begin{eqnarray}
 \tilde O_{0,\delta^\frac{1}{2}}(f) = \int_G h_1 (0,g) dg,\\
 \tilde O_{0,\eta\delta^\frac{1}{2}}(f) = \int_G h_2(0,g) dg.
\end{eqnarray}

Then it is easy to see:
\begin{lemma}
 For $f\in \mathcal W^+$, there are smooth functions $C_1, C_2$ so that in a neighborhood of zero:
\begin{eqnarray}
f(\xi) = |\xi|\left(-C_1(\xi)\ln|\xi| + C_2(\xi)\right)& \mbox{ in the split case,} \\
f(\xi) = |\xi|\left(C_1(\xi) + \eta(\xi) C_2(\xi)\right)& \mbox{ in the non-split case. }
\end{eqnarray}
Moreover, $C_1(0) =  \tilde O_{0,\delta^\frac{1}{2}}(f)$ and $C_2(0)=\tilde O_{u, \delta^\frac{1}{2}}(f)$ in the split case,
$C_1(0) = \tilde O_{0,\delta^\frac{1}{2}}(f)$ and $C_2(0) = \tilde O_{0,\eta\delta^\frac{1}{2}}(f)$ in the nonsplit case.
\end{lemma}

Thus, with the isomorphism of Corollary \ref{XinW} and the distributions defined in section \ref{sec:babycase}, we have:
\begin{eqnarray*} \tilde O_{0,\delta^\frac{1}{2}}(f) =  \tilde O_{0}(|\bullet|^{-1}f),\\
\tilde O_{u,\delta^\frac{1}{2}}(f) =  \tilde O_{u}(|\bullet|^{-1}f)
\end{eqnarray*}
in the split case, and:
\begin{eqnarray*}
 \tilde O_{0,\delta^\frac{1}{2}}(f) = \tilde O_{0,1} (|\bullet|^{-1} f),   \\
 \tilde O_{0,\eta\delta^\frac{1}{2}}(f) = \tilde O_{0,\kappa_0} (|\bullet|^{-1} f)
\end{eqnarray*}
in the nonsplit case.

\subsection{Inner product and limiting behavior at $\infty$} \label{ssipKuz}

Let now $Y$ be the complement of $(X\times X)^+$ in $X\times X$ (that is, the union of $\Gm$-translates of the diagonal copy of $X$), and denote by $\mathcal S(X\times X, \mathcal L_\psi\boxtimes \mathcal L_\psi^{-1})_Y$ the stalk of $\mathcal S(X\times X, \mathcal L_\psi\boxtimes \mathcal L_\psi^{-1})$ over $Y$. Notice that at this point we have restricted our attention to sections of our cosheaves over $X\times X$; that is, standard Schwartz sections of $\mathcal L_\psi\boxtimes \mathcal L_\psi^{-1}$ on $X\times X$. We denote by $\mathcal S(Y, \mathcal L_\psi\boxtimes \mathcal L_\psi^{-1})$ the fiber of $\mathcal S(X\times X, \mathcal L_\psi\boxtimes \mathcal L_\psi^{-1})$ over $Y$ -- it is the space of $\mathcal L_\psi\boxtimes \mathcal L_\psi^{-1}$-valued Schwartz functions on $Y$. Just for this subsection, we introduce the notation $\mathcal S(\mathcal W^0)$ for the $G$-coinvariants of $\mathcal S(X\times X, \mathcal L_\psi\boxtimes \mathcal L_\psi^{-1})$.

For $\Phi_1\in\mathcal S(X,\mathcal L_\psi)$, $\Phi_2\in \mathcal S(X,\mathcal L_\psi^{-1})$ and a measure $dx$ on $X$, we define the inner product:
$$\left<\Phi_1,\Phi_2\right> = \int_X (\Phi_1\cdot \Phi_2)(x) dx,$$
i.e.\ as a bilinear map. Clearly, it extends to a linear functional on $\mathcal S(X\times X, \mathcal L_\psi\boxtimes \mathcal L_\psi^{-1})$, and for $\Phi$ in this space we will be using the notation $\left<\Phi\right>$.

Given $f \in \mathcal S(\mathcal W^0)$, choose a pair $(\Phi,dg)$ consisting of an element $\Phi\in \mathcal S(X\times X, \mathcal L_\psi\boxtimes \mathcal L_\psi^{-1})$ and a Haar measure on $G$ so that $f$ is obtained as the coinvariants of $\Phi$ with respect to this measure. The chosen measure on $G$ induces a measure on $X$ as follows: let $x\in X$ and $N=G_x$, the stabilizer of $x$; hence, $X=N\backslash G$. The group $N$ acts by a character $\Psi$ on the fiber of $\mathcal L_\psi$ over $x$, and we choose an identification of $N(F)$ with $F$ such that the character $\Psi$ becomes our fixed additive character $\psi$; we then let $dn$ be the Haar measure on $N$ corresponding to our fixed measure $dx$ of \S \ref{ssTamagawa}, and we let $dx$ be the measure on $X$ corresponding to $dg, dn$. Clearly, it does not depend on the choice of point.

We then define the \emph{``inner product''} of $f$ to be the functional:
\begin{equation}\label{Kuzip}
 \left<f\right> = \left<\Phi\right>,
\end{equation}
where the ``inner product'' of $\Phi$ is defined with respect to the measure described above.

The following is immediate:
\begin{lemma}
 The inner product spans the space of $G$-invariant functionals on $\mathcal S(Y, \mathcal L_\psi\boxtimes \mathcal L_\psi^{-1})$ (the fiber of $\mathcal S(X\times X, \mathcal L_\psi\boxtimes \mathcal L_\psi^{-1})$ over $Y$).
\end{lemma}

Based on Proposition \ref{propNakayama} now, the stalk of $\mathcal S(\mathcal W^0)$ at $\xi=\infty$ is generated over the stalk of smooth functions by an element with nonzero ``inner product''. 

\begin{proposition}\label{germs-Kloosterman}
The stalk of $\mathcal S(\mathcal W^0)$ at $\xi=\infty$ coincides with the set of germs of all functions $f$ of the following form:
\begin{itemize}
 \item  in the nonarchimedean case:
 \begin{equation}f(\xi) = C(\xi^{-1}) \cdot \int_{|x^2| = |\xi|} \psi(\xi x^{-1}-x) dx,\end{equation}
where $C$ denotes a(n almost) smooth function\footnote{Notice that this stalk is one dimensional; equivalently, the stalk is generated by such functions with $C$ constant. This can be seen by direct computation, or by showing that the stalk is generated by the images of locally constant functions ``upstairs''.} defined in a neighborhood of zero, with $C(0)=\left<f\right>$.
 \item in the archimedean case:
 \begin{equation} \label{Rgerms-Kloosterman}
         f(\xi) = \int F \left(\frac{\xi}{|x|^2+1}, x\right) \psi\left(\frac{\xi \bar x}{|x|^2+1} - x\right) dx,
 \end{equation}
where $F$ is a Schwartz function on $\Gm\times \PP^1$, with $F(-1,\infty) = \left< f\right>$.
\begin{equation} 
\end{equation}
\end{itemize}
\end{proposition}

\begin{remark}\label{remark-Kloosterman}
 It may not be clear at first from the above expressions, but it will become clear from the stationary phase analysis of the archimedean integrals in \S \ref{ssstationary} that the stalks are generated over the stalk of smooth functions by a single element, as they should. Thus, we could also write them as $C(\xi^{-1})$ times the same integral with $F$ replaced by any preferred function with $F(-1,\infty)=1$. (Jacquet has computed the integrals explicitly in \cite{JaKl2}, and his work will be the basis for the analysis of \S \ref{ssstationary}.)
\end{remark}

\begin{proof}
 We can easily see that the germ of $f$ can be written as an integral of the form:
$$ \int \Phi \left(\left(\begin{array}{cc} \xi \\ & 1 \end{array}\right) \left(\begin{array}{cc} 1 \\ x & 1 \end{array}\right)\right) \psi^{-1}(x) dx,$$
where $\Phi \in \mathcal S(N\backslash \PGL_2,\psi)$, with $\psi$ here denoting the character $\left(\begin{array}{cc} 1 & x \\ & 1 \end{array}\right)\mapsto \psi(x)$, the measures being the standard ones, and $\Phi\left(\left(\begin{array}{cc} & 1 \\ 1 \end{array}\right)\right) = \left<f\right>$.

Now we decompose in terms of the Iwasawa decomposition $G=NAK$, where $K = \PGL_2(\mathfrak o)$ in the nonarchimedean case, and $K$ is the ``standard'' $\SO(2)$ or $\operatorname{SU}(2)$ in the real and complex case, respectively. 

We can easily see that in the nonarchimedean case, for:
$$\Phi(nak) = \begin{cases}
                                                                      \psi(n), & \mbox{ if } a \in A(\mathfrak o)\\
									 0, &\mbox{ otherwise}
                                                                     \end{cases}$$
we get:
$$\int \psi(\xi x^{-1} -x) dx.$$
Since this particular $\Phi$ satisfies $\Phi\left(\begin{array}{cc} & 1 \\ 1 \end{array}\right)) = \left<f\right> =1 \ne 0,$ it generates the fiber = stalk of $\mathcal S(\mathcal W^0)$ over $\infty$. It is easy to see that for large $|\xi|$ only the $x$ with $|x^2|=|\xi|$ contribute (see the proof of Theorem \ref{matching}), and this gives the desired claim.

In the archimedean case, the Cartan decomposition reads:
$$ \left(\begin{array}{cc} \xi \\ & 1 \end{array}\right) \left(\begin{array}{cc} 1 \\ x & 1 \end{array}\right) =$$ $$ = \left(\begin{array}{cc} 1 & \frac{\xi \bar x}{\sqrt{|x|^2+1}} \\ & 1 \end{array}\right) \left(\begin{array}{cc} \frac{\xi}{\sqrt{|x|^2+1}} \\  & {\sqrt{|x|^2+1}} \end{array}\right)\left(\begin{array}{cc} \frac{1}{\sqrt{|x|^2+1}}& \frac{-\bar x}{\sqrt{|x|^2+1}} \\ \frac{x}{\sqrt{|x|^2+1}} & \frac{1}{\sqrt{|x|^2+1}} \end{array}\right),$$
and the matrix $\left(\begin{array}{cc} & 1 \\ 1 \end{array}\right)$ is obtained as the limit when $x\to \infty$, $\frac{\xi}{|x|^2+1}\to -1$, hence the claim.
\end{proof}

\subsection{Contribution of irregular locus and convergence of orbital integrals}

Let $S\subset \bar X \times X$ be the irregular locus of the map: (\ref{rationalmap}), that is: the set of points $(\bar x, x)$ with $\bar x\in \PP^1$ equal to the image of $x$ under the natural map: $X\to \PP^1$.

\begin{proposition}\label{propnoirregular}
 The embedding: 
$$\mathcal S((\bar X\times X)\smallsetminus S, \mathcal L_\psi\boxtimes \mathcal L_\psi^{-1}) \hookrightarrow \mathcal S(\bar X\times X, \mathcal L_\psi\boxtimes \mathcal L_\psi^{-1})$$
induces an isomorphism on $G$-coinvariants, that is: 
$$\mathcal S(\mathcal W) = \mathcal S((\bar X\times X)\smallsetminus S, \mathcal L_\psi\boxtimes \mathcal L_\psi^{-1})_G.$$
\end{proposition}

This proposition already implies that all the invariant distributions that we have defined on $\mathcal S((\bar X\times X)\smallsetminus S, \mathcal L_\psi\boxtimes \mathcal L_\psi^{-1})$ (regular and irregular orbital integrals, including the inner product) extend to the whole space; for later use, we mention the following (which is easy to see):

\begin{lemma}\label{stabilizing}
 In the nonarchimedean case, the regular orbital integrals of an element of $\mathcal S(\bar X\times X, \mathcal L_\psi\boxtimes \mathcal L_\psi^{-1})$ can be decomposed as: $$\int_{N\backslash G} \int_N^*\,\,\, , $$
where $\int_N^*$ is a stabilizing integral over large compact open subgroups of $N$. 
\end{lemma}

\begin{proof}[Proof of Proposition \ref{propnoirregular}]
If we fix the stabilizer $N^-$ of a point on $X$, and denote by $\psi^{-1}$ the character by which it acts on the fiber of $\mathcal L_\psi^{-1}$, the problem is easily reduced to that of finding $(N^-,\psi)$-equivariant distributions on the stalk $V$ of $\mathcal S(\bar X,\mathcal L_\psi)$ over the unique point $y$ of $\PP^1$ fixed by $N^-$.  The notation $\mathcal S(\bar X,\mathcal L_\psi)$ means similar asymptotics as in \S \ref{ssnonstandard}, not sections of $\mathcal L_\psi$ over $\bar X$, but it is easy to see that as an $N^-$-module this stalk has a filtration:
$$0\to W\to V\to W\to 0,$$
where $W$ is isomorphic to the stalk of smooth sections of $\mathcal L_\psi$ over $y$.

If $S'=\{y\}\subset \overline X$, in the notation of appendix \ref{app:cosheaves} the stalk $V$ of $\mathcal S(\bar X, \mathcal L_\psi)$ over $S'$ has a separated decreasing filtration by $V_n:= \overline{\mathcal J_{S'}^n V}$. Clearly, the group $N^-$ acts trivially on $\mathcal J_{S'}/\mathcal J_{S'}^2=$ the cotangent space of $y$, and hence also on $\mathcal J_{S'}^n/\mathcal J_{S'}^{n+1}=$ the $n$-th symmetric power of the cotangent space. Moreover, recall that $\mathcal L_\psi$ is the trivial line bundle over $\PP^1$. Therefore, there are no $(N^-,\psi)$-equivariant functionals on the $n$-th graded piece of this filtration, which is an image of (actually, isomorphic to):
$$\mathcal J_{S'}^n/\mathcal J_{S'}^{n+1} \otimes \mathcal S(\{y\}, \mathcal L_\psi).$$

\end{proof}

\section{Matching and the fundamental lemma.} \label{sec:matching}

\subsection{Matching}

\begin{theorem} \label{matching}
 The operator $|\bullet|\cdot \mathcal G$ gives rise to a topological isomorphism:
\begin{equation}
 \mathcal S(\mathcal Z) \xrightarrow{\sim} \mathcal S(\mathcal W), 
\end{equation}
which satisfies: 
\begin{equation}
 \left< |\bullet|\mathcal Gf\right> = \gamma^*(\eta,0,\psi) \left<f\right>,
\end{equation}
where $\left<\,\,\right>$ denotes the inner products defined in \S \ref{ssiptorus}, \ref{ssipKuz}, and $\gamma^*(\eta,0,\psi)$ denotes the leading term in the Taylor expansion of the gamma factor $\gamma(\eta,s,\psi)$ around $s=0$.
\end{theorem}

\proof
 We have short exact sequences:
\begin{equation}
 0\to \mathcal S(\mathcal X) \to \mathcal S(\mathcal Z) \to \mathcal S(\mathcal Z)/\mathcal S(X) \to 0
\end{equation}
and:
\begin{equation}
 0\to |\bullet| \mathcal S(\mathcal X) \to \mathcal S(\mathcal W) \to \mathcal S(\mathcal W)_\infty \to 0.
\end{equation}

The arrows on the left are closed embeddings and come from (\ref{Schwartzisom}), where we restrict only to sections of $\mathcal M(\mathcal Y_1)$ with smooth orbital integrals -- that is, we allow singularities only at $\xi=0$, not at $\xi = -1$; and from Corollary \ref{XinW}. 

We have already seen that $\mathcal G$ is an automorphism of $\mathcal S(\mathcal X)$, hence $|\bullet|\mathcal G$ is an automorphism between the subspaces of the above sequences. There remains to see that it induces isomorphisms of the quotients.

By Corollary \ref{corollaryFourier} and standard properties of Fourier transform, the germs at $\xi=0$ of elements of $\iota \mathcal F(\mathcal S(\mathcal Z))$ are precisely the germs of functions of the form $f_1(\xi) + \psi\left(\frac{1}{\xi}\right) h(\xi)$ with $f_1,h$ smooth. Moreover, we claim that for $\iota \mathcal F (f) \sim \psi\left(\frac{1}{\xi}\right) h(\xi)$ (where $\sim$ denotes equality of germs), we have: 
\begin{equation}\label{h0}
 h(0)= \gamma^*(\eta,0,\psi) \left<f\right>.
\end{equation}

It suffices to prove (\ref{h0}) for one element $f$ for which $\left<f\right>$ is nonzero. Recall that for (almost) every character $\chi$ of $F^\times$, considered as a tempered distribution on $k$ by meromorphic continuation according to Tate's thesis, we have a relation:
\begin{equation}
 \widehat{\chi(\bullet)}= \gamma(\chi^{-1},0,\psi) \cdot |\bullet|^{-1}\cdot \chi^{-1}(\bullet).
\end{equation}
Indeed, this is just a reformulation of the functional equation for zeta integrals; in what follows, we denote the obvious \emph{bilinear} (not hermitian) pairing by angular brackets, and use the exponent $\psi$ when Fourier transform is taken with respect to the character $\psi$, instead of $\psi^{-1}$ which is our standard convention. We denote Tate's zeta integral of a function $\phi\in \mathcal S(F)$ by $\zeta(\phi,\chi,s)$.
$$\left<\phi,\widehat{\chi}\right> = \left<{\widehat{\hat\phi}}^\psi,\widehat{\chi}\right> = \left<\hat\phi,\chi\right>= \zeta(\hat \phi,\chi, 1) = $$
$$ =\gamma(\chi^{-1},0,\psi) \zeta(\phi,\chi^{-1},0) = \gamma(\chi^{-1},0,\psi) \left<\phi,\chi^{-1}(\bullet)\cdot |\bullet|^{-1}\right>.$$

This implies that a function on $F$ which is equal to $\chi(\xi)$ in a neighborhood of zero (and Schwartz elsewhere) has Fourier transform which is equal\footnote{Asymptotically equal in the archimedean case, i.e.\ the quotient by the stated function tends to $1$. This is proven by an easy argument multiplying the character by a smooth cutoff function.} to $\gamma(\chi^{-1},0,\psi) |\xi|^{-1}\chi^{-1}(\xi)$ in a neighborhood of infinity (and Schwartz elsewhere). In particular, (\ref{h0}) holds for the nonsplit case $\eta\ne 1$.

For the split case, we can obtain the function $-\cdot \ln|\xi|$ as the limit of:
$$ \frac{1}{t} - \frac{|\xi|^t}{t}.$$
A function which is equal to this in a neighborhood of zero has Fourier transform which is equal to $-\frac{\gamma(1,-t,0)}{t} |\xi|^{-t-1}$ in a neighborhood of $\infty$, and in the limit $t\to 0$ we obtain $\gamma^*(1,-t,0) |\xi|^{-1}$.

There remains to show that Fourier transform gives a continuous surjection from the set of functions of the form $\psi\left(\frac{1}{\xi}\right) h(\xi)$ around $\xi = 0$ (and Schwartz otherwise) to $|\bullet|^{-1}$ times the germs of Kloosterman integrals described in Proposition \ref{germs-Kloosterman}. It will be an implicit byproduct of the proof that, if $C$ is as in the remark following Proposition \ref{germs-Kloosterman}, then $C(0)=h(0)$, hence $ \left< |\bullet|\mathcal Gf\right> = \gamma^*(\eta,0,\psi) \left<f\right>$. 

We perform this for the archimedean case in the next subsection. For the nonarchimedean case, let us say that $h=1_{\mathfrak o}$. Then:
$$\mathcal F\left( \psi\left(\frac{1}{\bullet}\right) h(\bullet)\right) (\xi) = \int_{\mathfrak o} \psi(x^{-1}-\xi x)dx.$$

For $|\xi|$ larger than $|\mathfrak p^{-2} \cdot \mathfrak c^2|$ (and larger than $1$), where $\mathfrak c$ denotes the conductor of $\psi$,  we claim that only the terms with $|x^2|= |\xi|^{-1}$ contribute. Indeed, set $u=x^{-1}$ and $v = \xi x$ and assume that $|u|>|v|$ (the case $|u|<|v|$ is identical). Then $u$ has norm larger than $|\mathfrak p^{-1}\mathfrak c|$, and as it varies in a ball of radius $|\mathfrak p^{-1}\mathfrak c|$ around some point $u_0$, $v$ varies in a ball of radius less or equal than $|\mathfrak c|$ around $v_0 = \xi u_0^{-1}$. Therefore:
$$\int_{u_0 + \mathfrak p^{-1} \mathfrak c} \psi (u - \xi u^{-1} ) du = \int_{u_0 + \mathfrak p^{-1} \mathfrak c} \psi(u) du =0.$$

Hence,
$$\mathcal F\left( \psi\left(\frac{1}{\bullet}\right)1_{\mathfrak o}\right) (\xi) = \int_{|x|^2 = |\xi|^{-1}} \psi(x^{-1}-\xi x) dx = $$ 
$$ =|\xi|^{-1} \int_{|x|^2 = |\xi|} \psi(\xi x^{-1}-x) dx  .$$

\subsection{Stationary phase} \label{ssstationary}

We complete the proof of matching in the archime\-dean case, based on the arguments of \cite{JaKl2}. We only discuss the real case, as the complex case can be treated similarly.

\begin{lemma}\label{stationaryphase}
 Let $F=\RR$ and let $\phi(u,\delta)$ be a Schwartz function in two variables. The integral: 
$$ \int \phi(u,\frac{1}{\lambda}) \psi(\lambda(u+u^{-1})) du$$
is equal to $f_1(\lambda)+ |\lambda|^{-\frac{1}{2}}\psi(2\lambda) \theta_+ \left(\frac{1}{\lambda}\right) + |\lambda|^{-\frac{1}{2}}\psi(-2\lambda) \theta_-\left(\frac{1}{\lambda}\right)$, 
where $f_1$ is a Schwartz function of $\lambda$, and $\theta_\pm$ are smooth functions (supported in a neighborhood of zero) whose derivatives at zero are polynomials, without constant terms, on the derivatives of $\phi(u,\delta)$ at $u=\pm 1$ (respectively), $\delta=0$. In particular, $\theta_\pm(0)$ depends only on $\phi(\pm 1,0)$ (respectively).

Moreover, in the special case that $\phi(u,\delta) = f(u\delta)$ for some smooth function $f$, each derivative of $\theta_+$ at $0$ depends on a finite number of derivatives of $f$ at $0$, and the germ of $\theta_+$ at zero can be arbitrary. Similarly for $\theta_-$.
\end{lemma}

\begin{proof}
This is \cite{JaKl2}[Proposition 1], except for the last statement.

It is proven in \cite{JaKl2} that, up to a certain nonzero constant:
$$\theta_+ (\delta) = \int \phi_1(u,\delta) \psi\left(-\frac{u^2 \delta}{4}\right) du,$$
where $\phi_1$ is the partial Fourier transform in the variable $v = \frac{u-1}{\sqrt{u}}$ of the function:
$$\phi(u(v), \delta) \frac{du}{dv}.$$
(We assume without loss of generality that $\phi$ is supported close to $u=1$, so that the change of variables $v = \frac{u-1}{\sqrt{u}}$ is valid.)

Hence, 
$$\theta_+^{(n)}(\delta) = \int \left[\left(\frac{\partial}{\partial\delta} - \frac{2\pi i u^2}{4}\right)^n \phi_1(u,\delta)\right] \psi\left(-\frac{u^2 \delta}{4}\right) du.$$
(We assume without loss of generality that $\psi(x) = e^{2\pi i x}$.)

Therefore: $$\theta_+^{(n)}(0) = \left.\left(\frac{\partial}{\partial\delta} - \frac{1}{8\pi i} \frac{\partial}{\partial v}\right)^n \phi(u(v),\delta) \frac{du}{dv}\right|_{v=\delta =0}.$$

It is clear that, if $\phi = f(u\delta)$, this expression is bounded by a finite number of derivatives of $f$ at $0$, and that the evaluation of $\theta_+^{(n)}(0)$ involves higher derivatives of $f$ at $0$ than the evaluation of all $\theta_+^{k}(0)$, $k<n$. Therefore, the map $f\mapsto \theta_+$ is surjective onto the stalk of smooth functions at zero.
\end{proof}

This allows us to complete the proof of Theorem \ref{matching} in the real case: Indeed, for the by the stationary phase method or the arguments of \cite{JaKl2} it is easy to see that (\ref{Rgerms-Kloosterman}) is a Schwartz function of $\xi$ for $\xi>0$ or $F$ supported away from $-1$ (in the first variable). For $F$ supported close to $-1$ and $\xi<0$ we can make first the change of variables: $t	= \frac{x^2+1}{x}$, and the integral (\ref{Rgerms-Kloosterman}) becomes: 
$$ \int F_1 \left(\frac{\xi}{t^2}, t\right) \psi(\xi t^{-1}-t) dt,$$
where $F_1$ is another (arbitrary) smooth function on $\Gm\times \PP^1$. Then we can make the change $u= -\sqrt{-\xi}^{-1} t$ to turn this into:
\begin{equation}\label{Kythnos1}
  \sqrt{|\xi|} \int F_1 \left(- u , t\right) \psi(\sqrt{-\xi} (u^{-1}+u)) du.
\end{equation}

Similarly, for the Fourier transform of a function of the form $h(x) \psi\left(\frac{1}{x}\right)$ we have:
$$ \int h(x) \psi(x^{-1} - \xi x) dx,$$
which again by the same arguments depends up to a Schwartz function of $\xi$ only on the restriction of $h$ in a neighborhood of zero, and only for $\xi<0$. By the change of variables $u = -\sqrt{-\xi} x$ we get: 
\begin{equation}\label{Kythnos2} \sqrt{|\xi|}^{-1} \int h_1(-\sqrt{-\xi}^{-1}u) \psi(\sqrt{-\xi}(u+u^{-1})) du,
\end{equation}
where $h_1$ is another (arbitrary) smooth function in a neighborhood of zero.

By the last statement of Lemma \ref{stationaryphase}, the stalks at zero of (\ref{Kythnos1}), (\ref{Kythnos2}) coincide. This completes the proof of Theorem \ref{matching}.

\qed

\subsection{Basic vectors} \label{ssbasic}

From now on, until the end of this section, we assume that $F$ is nonarchimedean, $E$ (and hence $F$) is unramified over the base field $\QQ_p$ or $\mathbb F_p((t))$, and endow the groups $G,T,N, N^-$ with smooth group scheme structures over the ring of integers $\mathfrak o$. We set $K= G(\mathfrak o)$, a hyperspecial maximal compact subgroup. The conductor of our fixed self-dual character $\psi$ is equal to the ring of integers of $F$. We consider the $\mathfrak o$-schemes:
$$ X_1 = T\backslash G, \,\, X_2 = N\backslash G,$$
where the latter is equipped with the line bundle $\mathcal L_\psi$ defined by $\psi$ and an $\mathfrak o$-identification: $N\simeq \Ga$.

We endow the various groups with invariant differential forms defined over $\mathfrak o$, which are nonzero when reduced to the residue field. Based on our fixed measure on $F$ of \S \ref{ssTamagawa}, this gives rise to invariant measures on their $F$-points, and the $F$-points of their quotients; these measures are canonical, as they do not depend on the choice of differential form.

We consider the spaces $\mathcal S(\mathcal Z)$ and $\mathcal S(\mathcal W)$ of coinvariants corresponding to $X_1$, resp.\ $X_2$, as defined previously. We will define distinguished vectors $f_{\mathcal Z}^0$, $f_{\mathcal W}^0$ on them, the \emph{basic vectors}.

For $\mathcal S(\mathcal Z)$ we define:
\begin{equation}\label{basic-torus}
f_{\mathcal Z}^0:= \mbox{ the image of $1_{X_1(\mathfrak o)}\otimes 1_{X_1(\mathfrak o)}$ in }\mathcal S(\mathcal Z). 
\end{equation}
(Having fixed measures on the various groups, this image is a well-defined element of $\mathcal S(\mathcal Z)$.)

The description for $f_{\mathcal W}^0$ will be more complicated, as it is a ``nonstandard'' test function, i.e.\ not compactly supported. Recall from \ref{ssorbital-Kuz} that, in order to define orbital integrals for the Kuznetsov trace formula as functions, we have chosen a point $(x_0,y_0)\in (X_2\times X_2)^+$ with image $1\in \mathcal B$, and have trivialized the fiber of $\mathcal L_\psi\boxtimes \mathcal L_{\psi}^{-1}$ over that point. We now assume that $(x_0,y_0) \in (X_2\times X_2)^+(\mathfrak o)$, hence after trivializing the fiber and choosing suitable $\mathfrak o$-isomorphisms of the stabilizers with $N,N^-$ (and of the latter with $\Ga$) our sections become elements of
$C^\infty(N\backslash G\times N^-\backslash G, \psi \otimes \psi^{-1})$. In fact, we may trivialize both the fibers of $\mathcal L_\psi$ over $x_0$ and $\mathcal L_\psi^{-1}$ over $y_0$, to consider smooth sections of $\mathcal L_\psi$ (resp.\ of $\mathcal L_\psi^{-1})$ as elements of $C^\infty(N\backslash G,\psi)$ (resp.\ $C^\infty(N^-\backslash G,\psi^{-1})$).

For $n\in \mathbb N$, we denote by $1_{x_nK}$ the section:
\begin{equation}
 1_{x_nK} \left( u \left(\begin{array}{cc} \varpi^m \\ 1\end{array}\right) k\right) \,(\mbox{where }u\in N,k\in K) = \begin{cases} 0, & \mbox{ if } m\ne n, \\ \psi(u),& \mbox{ otherwise,} \end{cases}
\end{equation}
of $\mathcal L_\psi$. As $n$ varies in $\mathbb N$, these form a basis for the space of compactly supported, $K$-invariant sections of $\mathcal L_\psi$. We similarly define $1_{y_nK}^-$ for $\mathcal L_\psi^{-1}$.

For an algebraic representation $V$ of the dual group $\check G = \SL_2$, denote by $h_V$ the element of the spherical Hecke algebra $\mathcal H(G,K)$ corresponding under the Satake isomorphism:
$$\mathcal H(G,K) = \CC[\Rep(\check G)]$$
to the representation $V$. Here the monoid of dominant weights of $\check G$ is isomorphic to $\mathbb N$, and we will be writing $h_n$ for $h_{V_n}$, where $V_n$ is the $n$-th highest weight representation.

The Casselman-Shalika formula states that:
\begin{equation}
 h_n\star 1_{x_0K} = q^{-\frac{n}{2}}1_{x_nK}.
\end{equation}

Let $H_s$ be the formal series in the spherical Hecke algebra which \emph{corresponds under the Satake isomorphism to the $L$-function}: 
\begin{equation}L(\pi,\frac{1}{2}+ s)L(\pi\otimes\eta,\frac{1}{2}+s). 
\end{equation}
To understand what this means, we view an $L$-function $L(\pi,\rho,s)$ (where $\rho$ is a representation of the dual group) as a formal series (in the parameter $q^{-s}$) of traces of representations:
$$L(\pi,\rho,s) = \sum_{i=0}^\infty q^{-is}\tr(S^i\rho(\hat\pi)),$$
where $\hat\pi$ is the Satake parameter of $\pi$, hence the corresponding series in the Hecke algebra will be:
$$\sum_{i=0}^\infty q^{-is}h_{S^i\rho}.$$

We then define, for each $s$:
\begin{equation}
\Phi_s^0:= \Phi_{1,s}^0\otimes \Phi_2^0 = \left(H_s \star 1_{x_0K}\right)\otimes 1_{y_0K}^- \in C^\infty(X_2\times X_2,\mathcal L_\psi\otimes \mathcal L_\psi^{-1})^{K\times K}.
\end{equation}

To see that $H_s \star 1_{x_0K}$, a priori a formal series of elements of $C_c^\infty(X_2,\mathcal L_\psi)$, makes sense as a section of $\mathcal L_\psi$ when we fix $s$, write $H_s=h_{1,s}\star h_{2,s}$, where $h_{1,s}$ corresponds to the $L$-function $L(\pi,\frac{1}{2}+s)$ and $h_{2,s}$ corresponds to $L(\pi\otimes\eta,\frac{1}{2}+s)$.
Recall that for a representation $(\rho,V)$ of $\check G$, and $t\in \check G$, we have:
\begin{equation}
 \det^{-1}(I-q^{-s}\rho(t)|_V) = \sum_{n\ge 0} q^{-ns} \tr \rho(t)|_{S^nV}.
\end{equation}

Hence:
$$ h_{1,s} = \sum_{n\ge 0} q^{-n(s+\frac{1}{2})} h_n,$$
$$ h_{2,s} = \sum_{n\ge 0} q^{-n(s+ \frac{1}{2})} \epsilon^n h_n,$$
where $\epsilon = \pm 1$, according as $\eta$ is trivial or not.

Let $V_n$ denote the highest weight representation of $\check G$ corresponding to the $n$-th dominant weight, then we have the Klebsch-Gordan formula:
\begin{equation}
 V_m\otimes V_n = \sum_{l=0}^{\min(m,n)} V_{m+n-2l}.
\end{equation}

We use it to compute the convolution of $h_{1,s}$ with $h_{2,s}$, i.e.\ to write the series:
$$\left(\sum_{m\ge 0} q^{-m(s+\frac{1}{2})} h_m\right)\star\left(\sum_{n\ge 0} q^{-n(s+\frac{1}{2})} \epsilon^n h_n\right) = \sum_{n,m\ge 0} q^{-(n+m)(s+\frac{1}{2})} \epsilon^n h_m\star h_n=$$
$$= \sum_{n,m\ge 0} \sum_{l=0}^{\min(m,n)} q^{-(n+m)(s+\frac{1}{2})} \epsilon^n h_{m+n-2l}.$$

Let $k=m+n-2l$, then the restrictions between the different indices correspond to the system:
\begin{eqnarray*}
l\le \min(m,n) &\le& l+\frac{k}{2} \\ m+n&=&k+2l.
\end{eqnarray*}

To count all $m,n$ for a given $k$, we add over all $l=0,1,\dots$ and have two cases: either $m=\min(m,n)$, in which case $m$ ranges over: $l\le m\le \lfloor\frac{k}{2}\rfloor$; or $m>\min(m,n)$ in which case $n$ ranges over: $l\le n\le l+\lfloor\frac{k+1}{2}\rfloor$. Altogether, $n$ ranges from $l$ to $k+l$. Therefore, the coefficient for $h_k$ will be:
$$\sum_{l=0}^\infty q^{-(k+2l)(s+\frac{1}{2})} \sum_{n=l}^{k+l} \epsilon^n = \sum_{l=0}^\infty \epsilon^l q^{-(k+2l)(s+\frac{1}{2})}\cdot \begin{cases}
                                                                                                    k+1 &\mbox{ if }\epsilon=1,\\
0&\mbox{ if } \epsilon=-1, k\mbox{is odd},\\
1&\mbox{ if } \epsilon=-1, k\mbox{is even}.
                                                                                                   \end{cases}$$
Hence,
\begin{equation}\label{basicwhittaker}
 \Phi_{1,s}^0 = H_s \star 1_{x_0K} = \sum_{n=0}^\infty \frac{q^{-n(s+1)}}{1-\epsilon q^{-2s-1}}\cdot 1_{x_nK} \cdot \begin{cases}
                                                                                                    k+1 &\mbox{ if }\epsilon=1,\\
0&\mbox{ if } \epsilon=-1, k\mbox{is odd},\\
1&\mbox{ if } \epsilon=-1, k\mbox{is even}.                                             \end{cases}
\end{equation}

We deduce:

\begin{lemma}\label{lemmabasicOK}
 For each fixed $s$ such that $1-\epsilon q^{-2s-1}\ne 0$, $\Phi_{1,s}^0 $ makes sense as a smooth section of $\mathcal L_\psi$. Moreover, for $s=0$ we have: $\Phi_0^0 \in \mathcal S(\bar X\times X,\mathcal L_\psi\boxtimes \mathcal L_\psi^{-1})$. 
\end{lemma}

\begin{proof}
 Only the last assertion remains to be proven. We denote by $F_{1,s}^0$ the $K$-invariant \emph{function} on $X_2$ which, under the above trivializations, is equal to $\Phi_{1,s}^0$ on diagonal elements; that is, $F_{1,s}^0$ is given by the same series, but $1_{x_nK}$ is replaced by $1_{x_nK}':=$ the characteristic function of the $K$-orbit represented by $\diag(\varpi^n,1)$. Then it is easy to see that $\Phi_{1,s}^0$ is the product of $F_{1,s}^0$ by a section of $\mathcal L_\psi$ which extends to $\PP^1$. Therefore, it suffices to prove that $F_{1,0}^0$ satisfies, in the notation of \S \ref{ssnonstandard}:

$$ \left(1- \delta^{-\frac{1}{2}} (a) \mathscr L_a\right)\cdot \left(1- \eta_{E/F}\delta^{-\frac{1}{2}} (a)\mathscr L_a\right) F_{1,0}^0 = 0.$$

This follows immediately from the fact that $\mathscr L_{\diag(\varpi^m,1)} 1_{x_n K}' = q^\frac{m}{2} 1_{x_{n-m}K}'$.
\end{proof}

Therefore, we may define the basic vector:
\begin{equation}\label{basic-Kuz}
 f_{\mathcal W}^0:= \mbox{ the image of $\Phi_0^0$ in }\mathcal S(\mathcal W). 
\end{equation}

\subsection{Fundamental lemma}

Finally, we arrive at the ``fundamental lemma'' for elements of the Hecke algebra. Notice that the Hecke algebra $\mathcal H(G,K)$ does not act on the quotients $\mathcal S(\mathcal Z)$, $\mathcal S(\mathcal W)$. However, the Bernstein center does, since these are quotients of $G\times G$ representations (and we accept the convention that it is the Bernstein center of the first copy of $G$ which acts). The Bernstein center for the component of the spectrum corresponding to unramified principal series is isomorphic to $\mathcal H(G,K)$ under the natural map; therefore, we will abuse notation to write $h\star f$ for $h\in \mathcal H(G,K)$ and $f\in \mathcal S(\mathcal Z)$ or $\mathcal S(\mathcal W)$. Of course, this discussion serves only aesthetic purposes and is redundant otherwise, as we will only use such expressions for $f= $ the image of a $K\times K$-invariant function/section $\Phi_1\otimes \Phi_2$, and then $h\star f$ can be interpreted as the image of $(h\star \Phi_1)\otimes \Phi_2$.

\begin{theorem}\label{FL}
 For $f_{\mathcal Z}^0$, $f_{\mathcal W}^0$ the basic vectors defined in the previous subsection, and all $h\in \mathcal H(G,K)$, the integral transform $|\bullet|\mathcal G$ satisfies:
\begin{equation}
 |\bullet|\mathcal G \left( h\star f_{\mathcal Z}^0\right) = h\star f_{\mathcal W}^0.
\end{equation}
\end{theorem}

This could be proven by explicit calculations as follows: On one hand, one can explicitly compute the orbital integrals of characteristic functions of $K$-orbits (or the ``characteristic sections'' $1_{x_nK}\otimes 1_{y_mK}^-$ of the previous subsection); some of those computations are exhibited in section \ref{sec:explicit}. On the other hand, one can use a Casselman-Shalika type formula (which in this case is one of the easiest cases of the general formula computed in \cite{SaSph}) to explicitly describe the Hecke action in terms of those characteristic functions. 

Since this is tedious and not particularly informative, but mainly in order to demonstrate how the contents of the present paper are simply reflections, at the level of orbital integrals, of certain transforms taking place ``upstairs'' at the level of $G$-spaces plus prior work of Jacquet on the results of Waldspurger, we follow a shortcut; it is important to realize, though, that nothing in the present paper depends on the existence of this shortcut, as it could be done directly.

\begin{proof}
 
For the proof we will introduce intermediate ``spaces'' $\mathcal Z_1$ and $\mathcal W_1$ and we will prove ``fundamental lemmas'' for each step in the sequence:
$$ \mathcal S(\mathcal Z) \leftrightarrow \mathcal S(\mathcal Z_1) \leftrightarrow \mathcal S(\mathcal W_1) \leftrightarrow \mathcal S(\mathcal W).$$

Just for this proof, we write $F, F_1, G_1, G$ for the basic functions that have been defined, or will be defined, for the above spaces.

Symbolically, we have:
$$\mathcal Z_1 = A\backslash G/(A,\eta),$$
where $A$ is the split torus of diagonal elements and $\eta(\diag(a,1)) = \eta_{E/F}(a)$ -- in particular, $\mathcal Z_1 = \mathcal Z$ in the split case; and:
$$ \mathcal W_1 = (N,\psi)\backslash G/(N^-,\psi^{-1}),$$
but with different test functions than $\mathcal W$. 

More precisely, now, we define $\mathcal S(\mathcal Z_1)$ as the space of functions on $\mathcal B\smallsetminus\{0,-1\}$ obtained by orbital integrals of elements of the space:
$$ \mathcal S(A\backslash G\times A\backslash G,1\otimes \eta).$$
We use here the same parametrization for $A\backslash G/A$ as discussed in section \ref{sec:torus}, but since there is a nontrivial character $\eta$ we also need to specify representatives for the orbits which allow us to think of orbital integrals as functions on the regular set of $\mathcal B$. For $\Phi_1\otimes \Phi_2 \in \mathcal S(A\backslash G\times A\backslash G, 1\otimes \eta)$ we define:
\begin{equation}
 O_\xi (\Phi_1\otimes \Phi_2) = \int_G \Phi_1\left( \left( \begin{array}{cc} -\xi & 1+\xi  \\ -1 & 1 \end{array}\right) g \right) \Phi_2(g) dg.
\end{equation}

Throughout we assume smooth $\mathfrak o$-models for our groups, and Haar measures arising from residually nontrivial integral volume forms. The ``basic function'' is, of course, the image of $\Phi_1 =$ the characteristic function of $(A\backslash G)(\mathfrak o)$ and $\Phi_2 (a k) = \eta(a)$ for $a\in A, k\in K=G(\mathfrak o)$, $\Phi_2=0$ off $AK$. 

Jacquet has shown in Proposition 5.1 of \cite{JaW1} that there is a ``fundamental lemma for the Hecke algebra'' between $\mathcal Z$ and $\mathcal Z_1$, that is:
\begin{equation}
h\star F = h\star F_1
\end{equation}
for all $h\in \mathcal H(G,K)$ and $\xi\in \mathcal B^\reg_{\mathcal Z} = \mathcal B\smallsetminus\{0,-1\}$. The parametrization of orbits is different in loc.cit., as are the volumes, but in the end there is no need to normalize by volume factors -- as can easily be checked by taking $\xi\in \mathcal B_{\mathcal Z}^\reg(\mathfrak o)$.

Now we introduce the space $\mathcal S(\mathcal W_1)$, or rather just its basic vector $G_1$. This space will consist of the orbital integrals of certain smooth -- but not Schwartz -- sections of $\mathcal L_\psi\boxtimes \mathcal L_\psi^{-1}$ over $X\times X$, where $X=N\backslash G$. The basic vector $G_1$ will be obtained from the section $\Phi_1\otimes \Phi_2$, where $\Phi_1 = H_1\star 1_{x_0K}$ and $\Phi_2 = H_2 \star 1_{y_0K}^-$, in the notation of \S \ref{ssbasic}; Here $H_1$ and $H_2$ are the formal series in the Hecke algebra corresponding to the $L$-values:
$$ L(\pi,\frac{1}{2})$$
and 
$$ L(\pi\otimes\eta,\frac{1}{2}),$$
respectively. How to make sense of $\Phi_1$, $\Phi_2$ as sections is completely analogous to the discussion of \S \ref{ssbasic}. 

We claim that there is a ``fundamental lemma for the Hecke algebra'' between $\mathcal S(\mathcal W_1)$ and $\mathcal S(\mathcal W)$, that is:
\begin{equation}\label{brown}
 h\star G_1 = h\star G
\end{equation}
for all $h\in \mathcal H(G,K)$ and $\xi\in \mathcal B^\reg_{\mathcal W} = \mathcal B^\times$. It is convenient here to move to the domain of convergence by introducing a parameter $s$, i.e.\ functions $H_i^s$ defined as before, with the $L$-values taken at $\frac{1}{2}+s$ instead of $\frac{1}{2}$; we let $G_1^s$ the corresponding function of orbital integrals. Then, writing $H_1^s = \sum_n c(n,s) h_n$, $H_2^s = \sum_n d(n,s) h_n$, where $h_n$ is the Hecke element corresponding to the $n$-th dominant weight of the dual group, we have:
$$  h\star G_1^s(\xi) = \sum_{m,n} c(m,s) d(n,s) O_\xi (h\star h_m\star 1_{x_0 K}, h_n \star 1_{y_0K}^-)$$
for $\Re(s)$ large, by the fact that for such $s$ the regular orbital integrals are actual, convergent, integrals.

For an element $h$ in the full Hecke algebra of $G$, we denote by $h^\vee$ its linear dual: $h^\vee(g) = h(g^{-1})$. Elements of the spherical Hecke algebra of $\PGL_2$ are all self-dual. Since orbital integrals are invariant by the diagonal action of $G$, we get:
$$h\star G_1^s(\xi) = \sum_{m,n} c(m,s) d(n,s) O_\xi (h_n\star h\star h_m\star 1_{x_0 K}, 1_{y_0K}^-) = $$ $$ =O_\xi (H_2^s\star h \star H_1^s \star 1_{x_0 K}, 1_{y_0K}^-).$$

Finally, using the commutativity of the spherical Hecke algebra, this is equal to: $O_\xi (h\star H_1^s \star H_2^s \star 1_{x_0 K}, 1_{y_0K}^-) = h\star G^s$. Hence, $h\star G_1^s = h\star G^s$ for $\Re(s)$ large.

Taking the limit (analytic continuation) as $s\to 0$ we obtain (\ref{brown}). Taking the limit is justified as follows: on one hand, the sections $H_1\star 1_{x_0K}$ etc.\ are, by definition, pointwise limits of the sections $H_1^s \star 1_{x_0 K}$ etc. On the other, for given $\xi$ and sufficiently large $m$ \emph{or} sufficiently large $n$ we have $O_\xi\left((h_m 1_{x_0K})\otimes (h_n\star 1_{y_0K}^-)\right)=0$; this will be seen in \S \ref{ssorbital-characteristic}. 

We are left with showing the fundamental lemma for the passage $\mathcal S(\mathcal Z_1)\leftrightarrow \mathcal S(\mathcal W_1)$. To achieve that we will work on the level of spaces, and translate the ``unfolding'' method of Hecke to orbital integrals. 

Let $f_1 \in \mathcal S(A\backslash G,\delta^s)$, $f_1' \in \mathcal S(A\backslash G,\eta \delta^s)$. Recall that $\delta^s(\diag(a,1)) = |a|^s$. 

We define: 
$$f_2(g) = \int_N f_1(ng) \psi^{-1}(n) dn \in C^\infty(N\backslash G,\psi)$$ and:
$$f_2'(g)  = \int_N f_1'(nwg) \psi^{-1}(n) dn \in C^\infty(N^-\backslash G,\psi^{-1}),$$
where $w = \left(\begin{array}{cc} & 1 \\ -1 \end{array}\right)$.

We claim:
\begin{lemma}
If $f_1$ is the basic function of $A\backslash G$, i.e.\ the function supported on $AK$ with $f_1'(ak) = \delta^s(a)$, then $f_2 =  L(\bullet,\frac{1}{2}+s) \star 1_{x_0K}$, where by the $L$-value we mean the corresponding element of the Hecke algebra $\mathcal H(G,K)$. If $f_1'$ is the basic function of $(A\backslash G,\eta\delta^{s})$, i.e.\ the function supported on $AK$ with $f_1'(ak) = \eta\delta^{-s}(a)$, then $f_2' = L(\pi\otimes \eta, \frac{1}{2}+s)\star 1_{y_0K}^-$.
\end{lemma}

\begin{proof}
Assume that $\pi$ is an irreducible unramified representation and $W_\pi$ the spherical Whittaker function of $\pi$ with respect to $(N,\psi^{-1})$, normalized so that $W_\pi(1) =1$. Since by the Casselman-Shalika formula each $1_{x_n K}$ is up to a constant a multiple of $h_n\star 1_{x_0 K}$ we can write $f_2$ as a formal sum:
$$ \sum_{n=0}^\infty c(n) h_n\star 1_{x_0K},$$
so that if the integral: 
$$ \int_{N\backslash G} W_\pi(g) f_2(g) dg$$
is convergent, it is equal to: $$\sum_n c(n) \tr V_n(\hat \pi) \int W_\pi(g) 1_{x_0}(g) dg = \Vol(N\backslash G(\mathfrak o)) \sum_n c(n) \tr V_n(\hat \pi),$$
where $V_n$ is the $n$-th irreducible representation of the dual group and $\hat \pi$ the Satake parameter of $\pi$. Therefore we just need to compute this integral.

We write:
$$\int_{N\backslash G} W_\pi(g) f_2(g) dg = \int_G W_\pi(g) f_1(g) dg = $$
$$=\int_{A\backslash G} \int_A W_\pi(ag) \delta^{s}(a) da f_1(g) dg = \Vol(A\backslash G(\mathfrak o)) \int_A W_\pi(ag) \eta\delta^{s}(a) da.$$
It is well-known (and follows easily from the Casselman-Shalika formula) that the last integral is absolutely convergent for $\Re(s)> -\frac{1}{2}$, and equal to $\Vol(A(\mathfrak o)) L(\pi,\frac{1}{2}+s)$. This implies the claim, since $\frac{\Vol(A\backslash G(\mathfrak o))\Vol(A(\mathfrak o))}{\Vol(N\backslash G(\mathfrak o))} = \Vol(N(\mathfrak o))=1.$

This proves the lemma for $f_2$, and the proof for $f_2'$ is identical.
\end{proof}

We continue with the proof of Theorem \ref{FL}.  By the previous lemma, when $s=0$, we have $h\star F_1(\xi) = O_\xi (f_1\otimes f_1')$ and $G_1(\xi) = O_\xi(f_2 \otimes f_2')$ when $f_1 = h\star$(the basic function of $(A\backslash G,\delta^s)$) and $f_1' = $(the basic function of $(A\backslash G,\eta\delta^{s}$)).  We want to investigate the relationship between orbital integrals for $f_1\otimes f_1'$ and those for $f_2\otimes f_2'$, when $s=0$; as before, those will be the analytic continuation of the ones for $\Re(s)\gg 0$, where they are given by convergent integrals. We denote the Fourier transform of $f_1$ along $N$:
$$ \hat f_1(y, g) = \int_N f_1(ng) \psi^{-1}_y(n) dn,$$
where $\psi_y\left(\left(\begin{array}{cc} 1 & x \\ & 1 \end{array}\right)\right) = \psi(yx)$, and similarly for $\hat f_1'(y,g)$.

Clearly, $f_2(g) = \hat f_1 (1,g)$, $f_2'(g) = \hat f_1'(1,wg)$. Moreover, $\hat f_1(y,\diag(a,1)g) = |a|^{s+1} \hat f_1 (ay,g)$ and $\hat f_1'(y,w\diag(a,1)(g)) = \eta(a) |a|^{-s-1} \hat f_1'(a^{-1}y,g)$. Hence we have:
$$ O_\xi(f_2,f_2')  = \int_G \hat f_1(1, \diag(\xi,1) g) \hat f_1'(1,wg) dg = $$
$$ = |\xi|^{s+1} \int_{A\backslash G} \int_{F^\times} \hat f_1(a\xi, g) \hat f_1'(a^{-1}, wg) \eta(a) da dg$$

The function $\xi \mapsto \int_{F^\times} \hat f_1(a\xi, g) \hat f_1'(a^{-1}, wg) \eta(a) da$ can be seen as an orbital integral on $\Ga^2$ with respect to the action of the multiplicative group: $a\cdot (x,y) = (ax,a^{-1}y)$. Thus, we are in the split ``baby case'' of section \ref{sec:babycase}, except that we also have a character $\eta(a)$ in the orbital integrals. Moreover, we are applying those orbital integrals to the Fourier transform of a Schwartz function on $\Ga^2$ (indeed, the restrictions of $f_1,f_1'$ to unipotent orbits are Schwartz functions). We have then seen in \ref{sec:babycase} (for the case $\eta=1$, but the case $\eta\ne 1$ is similar) that:

\begin{equation*} \int_{F^\times} \hat f_1(a\xi, g) \hat f_1'(a^{-1}, wg) \eta(a) da = 
 \end{equation*}
\begin{equation}
 =\mathcal G\left( c\mapsto \int_{F^\times} f_1\left(\left(\begin{array}{cc}1 & c a\\ & 1 \end{array}\right)g\right) f_1'\left(\left(\begin{array}{cc}1 & a^{-1}\\  & 1 \end{array}\right)wg \right) \eta(a) da \right). 
\end{equation}

It follows that: 
$$ O_\xi(f_2,f_2') =  |\xi|^{s+1}  \cdot $$
$$ \cdot \mathcal G\left( c\mapsto  \int_{A\backslash G} \int_{F^\times} f_1\left(\left(\begin{array}{cc}1 & c a\\ & 1 \end{array}\right)g\right) f_1' \left(\left(\begin{array}{cc}1 & a^{-1}\\  & 1 \end{array}\right)wg \right) \eta(a) da dg\right) =$$
$$ = |\xi|^{s+1} \mathcal G\left( c\mapsto  \int_{A\backslash G} \int_{F^\times} f_1\left(\left(\begin{array}{cc}1 & c \\ & 1 \end{array}\right)\left(\begin{array}{cc}a^{-1} & \\ & 1 \end{array}\right)g\right)\right. \cdot  $$ 
$$ \cdot \left. f_1' \left(\left(\begin{array}{cc}1 & 1\\  & 1 \end{array}\right)\left(\begin{array}{cc} a & \\ & 1 \end{array}\right)wg \right) da dg\right) =$$
$$ = |\xi|^{s+1} \mathcal G\left( c\mapsto  \int_G f_1\left(\left(\begin{array}{cc}1 & c \\ & 1 \end{array}\right)g\right) f_1' \left(\left(\begin{array}{cc}1 & 1\\  & 1 \end{array}\right)wg \right) dg \right)   = $$
$$ = |\xi|^{s+1} \mathcal G\left( c\mapsto O_c (f_1\otimes f_1')\right).$$

This proves the theorem. 
\end{proof}

\section{Variation with a parameter and explicit calculations} \label{sec:explicit}

For global applications we will not be able to use the space of nonstandard sections for the Kuznetsov quotient directly. The reason is that, spectrally, they correspond to values of $L$-functions on the critical line, where global Euler products are non-convergent. We therefore need to introduce variations of this space, corresponding to the parameter $s$ in:
$$L(\pi,\frac{1}{2}+s) L(\pi\otimes \eta,\frac{1}{2}+s).$$

We conclude with this, and some explicit calculations.

\subsection{Nonstandard Whittaker space depending on $s$.} 

We generalize the definitions of \S \ref{ssnonstandard} to an arbitrary parameter $s\in \CC$ (the previous case corresponding to $s=0$), borrowing freely notation from there. \

We let $\mathcal M^s(\bar X\times X, \mathcal L_\psi\boxtimes \mathcal L_\psi^{-1})$ (resp.\ $\mathcal S^s(\bar X\times X, \mathcal L_\psi\boxtimes \mathcal L_\psi^{-1})$) denote the Schwartz cosheaf over $\bar X\times X$ consisting of smooth measures (resp.\ functions) on $X\times X$, valued in $\mathcal L_\psi\boxtimes \mathcal L_\psi^{-1}$, with the following properties:
\begin{itemize}
 \item the restriction of the cosheaf to $X\times X$ coincides with the standard cosheaf of Schwartz measures (resp.\ functions) valued in $\mathcal L_\psi\boxtimes \mathcal L_\psi^{-1}$; 
 \item in a neighborhood of $\mathbb P^1\times X$ they are finite sums of the form $\sum_i f_i F_i$, where:
\begin{enumerate}
 \item 
 the $f_i$'s are $\mathcal L_\psi\boxtimes\mathcal L_\psi^{-1}$-valued Schwartz functions on $\bar X\times X$;
 \item the $F_i$ are scalar-valued measures (resp.\ functions) on $X\times X$ which are $G$-invariant in the second coordinate, and in the first coordinate are annihilated asymptotically by the operator:
\begin{equation}\label{sasymptotics}
 \left(1- \delta^{-\frac{1}{2} - s} (a) \mathscr L_a\right)\cdot \left(1- \eta_{E/F}\delta^{-\frac{1}{2}-s} (a)\mathscr L_a\right).
\end{equation}
\end{enumerate}
\end{itemize}

We let $\mathcal M(\mathcal W^s)$ denote the $G$-coinvariants of $\mathcal M^s(\bar X\times X, \mathcal L_\psi\boxtimes \mathcal L_\psi^{-1})$. Again, using the trivializations of Lemma \ref{Kuztrivialization}, we have a map from $\mathcal M(\mathcal W^s)$ to measures on $\mathcal B^\times$. Finally, we identify those with functions on $\mathcal B^\times$, by dividing them by $|\xi|^{-2}d\xi$ (see the discussion following Lemma \ref{Kuzintegration}), and get the local Schwartz space $\mathcal S(\mathcal W^s)$ of the Kuznetsov trace formula with parameter $s$, consisting of functions on $\mathcal B^\times$. This is the space of orbital integrals of elements of $\mathcal S^s(\bar X\times X, \mathcal L_\psi\boxtimes \mathcal L_\psi^{-1})$.

\subsection{Basic vector} 

We now come to the setting of \S \ref{ssbasic}, adopting (until the end of the section) all the conventions and notation from there. In particular, $F$ is nonarchimedean and we have good integral models, measures, and isomorphisms for everything. We only denote here by $X$ what was denoted there by $X_2$; namely, the space $N\backslash \PGL_2$. We defined in \ref{ssbasic} certain sections $\Phi_s^0$ of $\mathcal L_\psi\boxtimes\mathcal L_\psi^{-1}$ over $X\times X$. In analogy with Lemma \ref{lemmabasicOK} we have:

\begin{lemma}
 The section $\Phi_s^0$ belongs to $\mathcal S^s(\bar X\times X, \mathcal L_\psi\boxtimes \mathcal L_\psi^{-1})$.
\end{lemma}

The proof is identical to that of Lemma \ref{lemmabasicOK}. We define $f_s^0$ to be the image of $\Phi_s^0$ in $\mathcal S(\mathcal W^s)$. (In comparison to \S \ref{ssbasic}, we omit the index $~_{\mathcal W}$ since here we only work on the Kuznetsov space, and introduce the index $~_s$ so that the previous $f_{\mathcal W}^0$ is now $f_0^0$.)

\subsection{Orbital integrals for the characteristic sections}\label{ssorbital-characteristic}

Recall that $1_{x_m K}$ denotes a certain compactly supported section of $\mathcal L_\psi$ defined in \S \ref{ssbasic}, and $1_{y_m K}^-$ a compactly supported section of $\mathcal L_\psi^{-1}$. Now we compute the orbital integral:
$$ O_{\xi} (1_{x_m K} \otimes  1_{y_0 K}^-) $$
for $\xi \in F^\times$. We also identify $\xi$ with the representative $\left(\begin{array}{cc} \xi \\ & 1\end{array}\right)$ of $N\backslash G/N^-$, according to \S \ref{ssorbital-Kuz}.

We have:
$$ O_{\xi} (1_{x_m K} \otimes  1_{y_0 K}) = \int_{N^-\backslash G} \int_{N^-} 1_{x_m K} (\xi n g) \psi^{-1}(n) dn \cdot 1_{y_0K}^-(g) dg =$$
$$ \Vol(X(\mathfrak o)) \int_{N^-} 1_{x_m K} (\xi n) \psi^{-1}(n) dn.$$

Let $n=\left(\begin{array}{cc} 1&  \\ x  & 1 \end{array}\right)\in N$, 
then $\xi n$ admits the following Iwasawa decomposition ($G=NAK$):
\begin{itemize}
 \item if $|x|\le 1$: then $\xi\in A$, $n\in K$;
 \item if $|x|>1$: then $\xi n = \left(\begin{array}{cc}  1 & \xi x^{-1} \\ & 1 \end{array}\right)\left(\begin{array}{cc} -\xi x^{-1}  &   \\ & x \end{array}\right)
\left(\begin{array}{cc}  & 1  \\ 1& x^{-1} \end{array}\right).$
\end{itemize}

Therefore, 
$$ \int_{N^-} 1_{x_m K} (\xi n) \psi^{-1}(n) dn = $$ $$ = 1_{x_m K} \left(\begin{array}{cc}  \xi &  \\ & 1 \end{array}\right) + \sum_{i=1}^\infty 1_{x_m K} \left(\begin{array}{cc}  \xi \varpi^i &  \\ & \varpi^{-i} \end{array}\right) \int_{\mathfrak p^{-i}\smallsetminus \mathfrak p^{-i+1}} \psi(\xi x^{-1} -x) dx.
$$

Thus we get:
\begin{itemize}
 \item if $|\xi|=q^{-m}$: $O_{\xi} (1_{x_m K} \otimes 1_{y_0K}^-)= \Vol X(\mathfrak o)$;
 \item if $|\xi|=q^{2i-m}$ for some $i>0$: 
$$O_{\xi} (1_{x_m K} \otimes 1_{y_0K}^-)= \Vol X(\mathfrak o) \int_{\mathfrak p^{-i}\smallsetminus \mathfrak p^{-i+1}} \psi(\xi x^{-1} -x) dx;$$
 \item zero otherwise.
\end{itemize}
For the integral in the second case, we have $|\xi x^{-1}| = |x| q^{-m} \le |x|$. Hence:
\begin{itemize}
 \item If $m\ge 1$ and $i>1$ then as $x$ varies in a ball of radius $q$, $\xi x^{-1}$ varies in a ball of radius $\le 1$, therefore the integral is zero.
 \item If $m\ge 1$ and $i=1$, i.e.\ $|\xi|=q^{2-m}$ then as $x$ varies in $\mathfrak p^{-1}\smallsetminus \mathfrak o$, $\psi(\xi x^{-1})=1$ and we get: $O_{\xi} 1_{x_m K} = -\Vol X(\mathfrak o)$.
 \item Finally, if $m=0$ and $|\xi|>1$ then we get (with a change of variables $x\mapsto -x$): $O_{\xi} 1_{x_m K} = \Vol X(\mathfrak o) \int_{|x|^2=|\xi|} \psi(x-\xi x^{-1}) dx$.
\end{itemize}

To summarize:

\begin{equation}\label{orbital-characteristic}
 O_{\xi} (1_{x_m K}\otimes 1_{y_0 K}^-) = \Vol X(\mathfrak o) \cdot
\begin{cases}
1 & \mbox{ if } |\xi|=q^{-m};\\
-1 & \mbox{ if } |\xi|=q^{2-m}, m\ge 1;\\
\int_{|x|^2=|\xi|} \psi(x-\xi x^{-1}) dx & \mbox{ if } |\xi|>1, m=1.
\end{cases}
\end{equation}
 
\begin{remark}
 We notice that for any $\xi$ and sufficiently large $m$ we have $O_\xi(1_{x_mK}\otimes 1_{y_0K}) = 0$. Thus, for an element $\Phi\in \mathcal S^s(\bar X\times X, \mathcal L_\psi\boxtimes \mathcal L_\psi^{-1})^{K\times K}$ which can be written as a series:
$$\Phi = \sum_{m\ge 0} c(m) (1_{x_mK}\otimes 1_{y_0K}^-),$$
a regular orbital integral $O_\xi(\Phi)$ can be written as an eventually stabilizing series (compare with Lemma \ref{stabilizing}):
$$ O_\xi(\Phi) = \sum_{m\ge 0} c(m) O_\xi (1_{x_mK}\otimes 1_{y_0K}^-).$$
\end{remark}

\subsection{Orbital integrals of the basic function}

Recall that the basic vector $f_s^0\in \mathcal S(\mathcal W^s)$ is obtained by the orbital integrals of $\Phi_s^0 = (H_s \star 1_{x_0K}) \otimes 1_{y_0K}^-$, where $H_s$ is the formal series in the Hecke algebra corresponding to the unramified $L$-factor $L(\pi,\frac{1}{2}+s) L(\pi\otimes \eta, \frac{1}{2}+s)$. We compute its regular orbital integrals, according to the previous remark.

\begin{lemma}\label{blue}
We have:
\begin{equation*}
 f_s^0(\xi) = O_\xi (H_s\star 1_{x_0K} \otimes 1_{y_0K}) = 
\end{equation*}
\begin{equation}
 = \Vol(X(\mathfrak o)) L(\eta,2s+1)\left(|\xi|^{s+1} \cdot (I-q^{-2s-1} \varpi^2\cdot) f(\xi) + 1_{|\xi|=q^2}  + \Kl(\xi)\right), 
\end{equation}
where:
\begin{itemize}
 \item $\Kl(\xi)$ denotes the function which is supported on $|\xi|>1$ and equal to: $\int_{|x|^2=|\xi|} \psi(x-\xi x^{-1}) dx$ there. ($\Kl$ stands for ``Kloostserman''.)
 \item $f$ is the function supported on $|\xi|\le 1$ and equal, there, to:
 $$ \begin{cases}
              1-\log_q|\xi| & \mbox{ in the split case,}\\
              \frac{1+\eta(\xi)}{2} & \mbox{ in the non-split case;}
             \end{cases}$$
 \item the action of $\varpi^2$ is normalized as in (\ref{unitaryaction}).
\end{itemize}
\end{lemma}

\begin{proof}
 Indeed, for given $\xi$ with $|\xi|=q^{-n}$, the first term expresses the contributions of $1_{x_nK}$ and $1_{x_{n+2}K}$ whenever those are nonzero, according to the first two cases of (\ref{basicwhittaker}). However, there is no contribution from $1_{x_0K}$ when $|\xi|=q^2$, and this is what the second term is correcting. The third term expresses the contribution of $1_{x_0K}$ when $|\xi|>1$. 
\end{proof}

 \appendix

\begin{appendices}

\section{Almost smooth functions, Schwartz and tempered functions} \label{app:almostsmooth}

\subsection{Almost smooth functions} The space of smooth functions on a real manifold has the structure of a Fr\'echet space. We would like, for the purpose of uniformity, to define a similar Fr\'echet space of functions for a $p$-adic manifold $X$, i.e.\ for a topological space equipped with an atlas of ``$p$-adic analytic functions'', which is locally isomorphic to the ring of $p$-adic analytic functions on $\mathfrak o^n$ (where, as usual $\mathfrak o$ denotes the ring of integers of a local nonarchimedean field $F$). The usual notions of ``locally constant'' and ``uniformly locally constant'' (when there is some uniform structure) functions do not lead to Fr\'echet spaces. We are going to define a new class of functions, which in this appendix will be called ``almost smooth'' and in the rest of the paper, for simplicity, just ``smooth''. Also, in this appendix we will be denoting the space of these functions by $C^\pinfty$, but in the rest of the paper just by $C^\infty$. Finally, for any statement 
about ``almost smooth'' functions in this appendix, when applied to real manifolds, the word ``almost'' should be disregarded; and moreover, complex manifolds and varieties will be considered as real manifolds/varieties.

Almost smooth functions will form a sheaf for the usual Hausdoff topology on $X$, and therefore it is enough to describe them locally around each point $x\in X$. 

We choose an analytic chart for a neighborhood $U$ of $x$, so that it becomes isomorphic to $\mathfrak o^n$ with its ring of analytic functions. Then we identify $C^\pinfty(U)$ with $C^\pinfty(\mathfrak o^n)$, the space of \emph{almost smooth} functions on $\mathfrak o^n$, defined as those complex-valued functions of the form: 
\begin{equation}\label{sumofsmooth}
 f = \sum_{i\ge 0} f_i
\end{equation}
on $\mathfrak o^n$, where $f_i$ is invariant under $\mathfrak p^i\times \cdots\times \mathfrak p^i$ and for every $N>0$ there is a scalar $C$ such that $\Vert f_i\Vert_\infty < C q^{-iN}$ for all $i$. It is a Fr\'echet space under any of the following equivalent systems of seminorms:

\begin{lemma}
 On the space of continuous functions on $\mathfrak o^n$, the following seminorms define \emph{tamely equivalent}\footnote{Recall that a tame Fr\'echet space is a Fr\'echet space with a presentation as an inverse limit of Banach spaces $B_n$, and a map $T:\lim_{\from} B_n\to \lim_{\from} B_n'$ is tame  if there are integers $b$ and $r$ so that for all $n\ge b$ the map $T$ is continuous from $B_{n+r}$ to $B'_n$.} Fr\'echet spaces:
\begin{enumerate}
 \item $\Vert f\Vert_N := \sup_{i\ge -1,x\in\mathfrak o^n} q^{iN} |f(x)-K_i\star f(x)|$;
 \item $\Vert f\Vert_N := \sup_{i\ge -1,x\in\mathfrak o^n} q^{iN} |K_{i+1} \star f(x)-K_i\star f(x)|$;
 \item $\Vert f\Vert_\infty$, and $\Vert f\Vert_N := \sup_{x\in\mathfrak o^n}\sup_{y\in \mathfrak o^n\smallsetminus\{0\}} |y|^{-N} |f(x) - f(x+y)|$  \\ (where $|(y_1,\dots,y_n)| = \sup_i |y_i|$). 
\end{enumerate}
Here $K_i$ is the characteristic measure of $\mathfrak p^i\times \cdots\times \mathfrak p^i$, convolution is in the additive group $\mathfrak o^n$, and by convention $K_{-1}\star f = 0$.
\end{lemma}

It is clear that the Fr\'echet structure is preserved under analytic automorphisms, thus the notion of an almost smooth function on a $p$-adic manifold is well-defined. 
Notice that, like smooth functions, these ``almost smooth'' functions have vanishing derivatives, for any reasonable notion of ``derivative'', for instance for any $Z\in \mathfrak o^n$ we have:
$$\lim_{t\to 0} \frac{f(tZ)-f(0)}{|t|} = 0.$$
Therefore, \emph{any statements about derivatives in the nonarchimedean case, throughout the paper, should be taken to concern only the zeroth derivative}. However, we will encode the issue of how fast a function varies in what we will call ``pseudo-derivatives'', a notion that is related to the seminorms defined above.

\subsection{Semialgebraic sets and charts} \label{ss:semialgebraic} We recall that a semialgebraic set on a real algebraic variety $X$ is obtained by a boolean combination (i.e.\ by taking unions and complements a finite number of times) of subsets of $X(\RR)$ given by an inequality of the form $f\ge 0$, where $f$ is a regular function. For a smooth algebraic variety $X$ over a nonarchimedean field $F$, on the other hand, semialgebraic sets are defined as boolean combinations of sets of the form:
$$ \{x\in X(F) | f(x) \in P_k\},\mbox{ where }P_k=\{y^k|y\in F\},$$
$f$ is a regular function, and $k\in \mathbb N_{\ge 2}$, cf.\ \cite{Denef-semialgebraic}.
By definition, a map: $X\to Y$ between semialgebraic sets is called semialgebraic if its graph is semialgebraic.

The above sets are the basic closed sets for the \emph{restricted topology} of semi-algebraic sets (restricted means: only finite unions of open sets are required to be open), and this is the topology we will be using when talking about ``open'' and ``closed'' sets and neighborhoods, unless otherwise specified.

Notice that, in general, the notion of closure is not well-behaved for restricted topologies. However, for semialgebraic sets the following is true: The closure of a semialgebraic set in the usual (Hausdoff) topology is closed semialgebraic; hence, the notion of closure is well behaved, and closure in semialgebraic topology coincides with closure in the Hausdorff topology.

By a smooth semialgebraic set we will mean an open subset of the points of a smooth variety. (One can more generally define ``semialgebraic manifolds'', but we will not need this.) For the description of tempered functions, we will need to introduce a notion of ``semialgebraic chart'' for smooth semialgebraic set $U$ in the nonarchimedean case. By a \emph{semialgebraic chart} of $U$ we mean a finite partition into open-closed subsets: $U=\sqcup_j U_j$ and, for every $j$, a semi-algebraic isomorphism $\alpha_j: V_j\xrightarrow{\sim} U_j$ with an open semialgebraic subset of $F^n$ \emph{which is $\mathfrak o^n$-stable} (under addition in $F^n$).

\begin{lemma}
 Semialgebraic charts exist.
\end{lemma}

\begin{proof}
 It is easy to see that any smooth semialgebraic set in the nonarchimedean case is isomorphic to a finite disjoint union of open semialgebraic subsets of $F^n$, so it remains to consider the case that $U\subset F^n$, in order to show that one can find an $\mathfrak o^n$-invariant chart. 

 For simplicity, we only show that this is the case for the basic open set $\{x|f(x)\notin P_k\}$ where $P_k$ is the set of $k$-th powers of elements of $F$ as above and $f$ is a polynomial; the general case is only notationally more complicated. We may even restrict to the intersection of this set with $\mathfrak o^n$, by partitioning the set and inverting coordinates as appropriate. Away from any neighborhood of the zero set of $f$ (in $\mathfrak o^n$) the condition: $f(x)\notin P_k$ is locally constant in $x$, hence uniformly locally constant, hence multiplying the coordinates by a suitable scalar will give the required chart. We are left with finding an $\mathfrak o^n$-invariant chart for a neighborhood of the zero set $Z\subset \mathfrak o^n$ of $f$. By a resolution of singularities (which will be recalled in the next appendix), we can replace a neighborhood of $Z$ by a compact semialgebraic set $V$ of the same dimension, so that the pullback of $f$ is, in semialgebraically-local coordinates $(y_1,\dots,
y_n)$, 
of the form: $c\cdot y_1^{i_1}\cdots y_n^{i_n}$. Then the claim is easy to show.
\end{proof}

\subsection{Schwartz functions} \label{ss:schwartzfunctions}

If $U$ is an open semialgebraic subset of (the points of) a real or $p$-adic variety, we will define the space $\mathcal S(U)$ of \emph{Schwartz functions} on $U$ as a space of smooth (in the archimedean case), resp.\ almost smooth functions (in the nonarchimedean). The definition in the archimedean case is well-known, but to construct its analog for the nonarchimedean we need to take into account not only the growth of $f$, but also the growth of the summands $f_i$ of an expression as in (\ref{sumofsmooth}). For this, we will introduce the following analog of differential operators:

\begin{definition}
 Let $U$ be an open subset of the points of a smooth $p$-adic variety and $\mathcal C:=(U_i,V_i,\alpha_i)_i$ a semialgebraic chart of $U$. For each almost smooth function $f$ on $U$ and each $N\ge 0$ we define the \emph{$N$-th pseudoderivative of $f$ with respect to $\mathcal C$} to be equal to $f$ if $N=0$, and otherwise:
\begin{equation}
 f^{\mathcal C,(N)} (x) = \sup_{y\in \mathfrak o^n\smallsetminus\{0\}} |y|^{-N} |f(x) - f(x+y)|,
\end{equation}
where the ``sum'' $x+y$ should be interpreted in terms of the chart $\mathcal C$ -- i.e.\ it really means: $\alpha_i(\alpha_i^{-1}x+y)$ for $x\in U_i$.
\end{definition}

It is easy to prove:
\begin{lemma}
 If $\mathcal C$, $\mathcal C'$ denote two different charts, for every $N$ there is a semialgebraic function $T$ such that:
$$ |f^{\mathcal C, (N)}(x)| \le |T(x)| \cdot |f^{\mathcal C', (N)}(x)|.$$
\end{lemma}
As a corollary, the notions of Schwartz and tempered functions that we are about to define do not depend on the choice of chart; we will omit the chart from the notation for pseudoderivatives from now on.

Now we define Schwartz functions on $U$. We recall that a ``Nash differential operator'' is a ``smooth semialgebraic'' differential operator, cf.\ \cite{AG-Nash}. In particular, the growth of these operators is bounded, locally for the semialgebraic topology, by regular functions.

\begin{definition}
 The space $\mathcal S(U)$ of Schwartz functions on $U$ consists of those smooth functions on $U$, in the archimedean case, resp.\ almost smooth in the nonarchimedean case, with the property:
\begin{itemize}
 \item for every Nash differential operator $D$ on $U$, in the archimedean case, and for every (equivalently: some) chart $\mathcal C$, every $N\ge 0$ and semialgebraic function $T$, in the nonarchimedean case, the function $Df$, resp.\ $T f^{(N)}$, is bounded.
\end{itemize}
\end{definition}

The space of Schwartz functions on $U$ is naturally a \emph{nuclear Fr\'echet algebra}; its topology is generated by the seminorms:
$$\sup_{x\in U} |Df(x)|,$$
in the archimedean case, where $D$ varies over all Nash differential operators (evidently, a countable number of them suffices), and:
$$\sup_{x\in U} |T(x)f^{(N)}(x)|$$
in the nonarchimedean, where $N\in \mathbb N$ and $T$ varies over all semialgebraic functions on $U$ (again, a countable number suffices).

We will discuss in appendix \ref{app:cosheaves} cosheaf-theoretic properties of Schwartz functions. The following will be a consequence of \ref{lemmaKS}:

\begin{proposition}
 Schwartz functions on $U$ are precisely those functions which, for one, equivalently any, smooth compactification $\bar U$ of $U$ extend to smooth functions (in the archimedean case) resp.\ almost smooth functions (in the non-archimedean case) all of whose derivatives vanish on $\bar U\smallsetminus U$.
\end{proposition}
We remind that any statement about derivatives should be understood to apply only to the zeroth derivative in the nonarchimedean case.

\subsection{Tempered functions} \label{ss:temperedfunctions}

\begin{definition}
If $U$ is an open semialgebraic subset of (the points of) a smooth real or $p$-adic variety, we define the space $\mathcal O(U)$ of \emph{tempered functions} on $U$ as those smooth (in the archimedean case), resp.\ almost smooth (in the nonarchimedean) functions $f$ on $U$ with the property:
\begin{itemize}
 \item In the archimedean case, for every Nash differential operator $D$ on $U$ there is a semialgebraic function $T$ on $U$ with $|Df|\le |T|$; in the nonarchimedean, for every (equivalently: one) chart $\mathcal C$ and any $N\in \mathbb N$ there is a semialgebraic function $h$ on $U$ with $|f^{(N)}|\le |T|$.
\end{itemize} 
\end{definition}

The space $\mathcal O(U)$ of tempered functions on $U$ is an algebra which acts on the space of Schwartz functions:
$$ \mathcal O(U)\otimes \mathcal S(U)\to \mathcal S(U).$$

Moreover, each $f\in \mathcal O(U)$ is a bounded operator on $\mathcal S(U)$. We endow $\mathcal O(U)$ with the \emph{strong operator topology} on the Fr\'echet space $\mathcal S(U)$; this way it becomes a locally convex topological algebra. By definition, convergence to zero of a net $(f_\alpha)_\alpha \subset \mathcal O(U)$ in the strong topology means that $f_\alpha\phi\to 0$ for every $\phi \in \mathcal S(U)$. Since $\mathcal S(U)$ is a nuclear, and hence Montel Fr\'echet space (i.e.\ bounded sets are precompact), it is known that this topology coincides with the operator topology of uniform convergence on bounded/compact sets \cite[p.\ 139]{Koethe}. It is easy to describe sequential convergence in this topology:

\begin{lemma}\label{sequentialconvergence}
 For a sequence $f_n\in\mathcal O(U)$ we have $f_n\to 0$ iff:
\begin{enumerate}
 \item $f_n\to 0$ in $C^\infty(U)$ (with the usual Fr\'echet topology of locally uniform convergence of all derivatives), resp.\ in $C^\pinfty(U)$, and
 \item for each Nash differential operator $D$, resp.\ for each chart $\mathcal C$ and integer $N$, there is a semialgebraic function $T$ such that:
$$ |Df_n|\le |T|,\,\,\mbox{ resp.\ }|f_n^{(N)}|\le |T| \mbox{ for all }n.$$
\end{enumerate}
\end{lemma}

\begin{proof}
 It is clear that such a sequence is a null sequence. Vice versa, it is clear that a null sequence should converge to zero in $C^\infty(U)$, resp.\ $C^\pinfty(U)$. The proof of the second condition is reduced to tempered functions on $F$ by the resolution of singularities that will be recalled in the next appendix (Theorem \ref{resolution}). We prove that all $f_n$ should be bounded by some $|x|^{N}$, for some $N$, in a neighborhood of $\infty$ (the proof for derivatives/pseudoderivatives is similar): if not, there is a Schwartz function $\phi$ on $F$ with $\sup_x |f_n(x)\phi(x)|$ bounded below. Thus, $f_n$ cannot be a null sequence.
\end{proof}

In particular, $\mathcal O(U)$ is sequentially complete. In fact, it can be shown that it is a complete, nuclear topological vector space, but we will not use this.

\section{Schwartz cosheaves} \label{app:cosheaves}

In this appendix we formalize certain properties of Schwartz functions. These properties are obvious for the Schwartz functions themselves, but not totally obvious for their coinvariants, hence the language that we are introducing is helpful in analyzing orbital integrals. 

From now on, as in the main body of the text, \textbf{``smooth'' function means ``almost smooth'' at nonarchimedean places}. Throughout this section, $X$ denotes the $F$-points of a smooth algebraic variety over a local field $F$, ``closed'' and ``open'' refer to the restricted topology of semialgebraic sets.

Finally, to avoid repeating the same dichotomy again and again, any mention of \textbf{``a Nash differential operator $D$'' should be understood, in the nonarchimedean case, as} the data consisting of:
\begin{enumerate}
 \item a semialgebraic chart $\mathcal C$;
 \item an integer $N\ge 0$;
 \item a semialgebraic function $T$.
\end{enumerate}
Then, a statement about the function $Df$ should be replaced by the analogous statement about the function $|T|f^{(N)}$, as in the definition of Schwartz functions in Appendix \ref{app:almostsmooth}. On the other hand, we keep our convention that \textbf{any statement about derivatives should be understood to apply only to the zeroth derivative in the nonarchimedean case}. By consistently using the phrases ``Nash differential operator'' and ``derivative'', this should cause no confusion.

\subsection{The sheaf of tempered functions}

The association $U\to \mathcal O(U)$, where $\mathcal O(U)$ denotes the space of tempered functions on $U$ (\S \ref{ss:temperedfunctions}), is a sheaf of topological algebras on $X$. We will consider it as the ``structure sheaf'', in the sense that all other sheaves will be modules for it.

Except for the general sheaf properties, what is interesting for us now is the following relation between topology and algebra structure: Any closed $S\subset X$ gives rise to the sheaf of ideals $\mathcal J_S\subset \mathcal O$ of functions vanishing on $S$. We denote by $\mathcal J_S^n(U)$ the \emph{closed} ideal of $\mathcal O(U)$ generated by $n$-fold products of elements in $\mathcal J_S(U)$; we write $\mathcal K_S (U) = \cap_n \mathcal J_S^n(U)$. Both $\mathcal J_S^n$ and $\mathcal K_S$ are sheaves on $X$.

\begin{lemma}\label{lemmaKS}
 The sheaf $\mathcal K_S$ is the sheaf of tempered functions which vanish, together with all their derivatives, on $S$; in particular, in the nonarchimedean case $\mathcal K_S = \mathcal J_S$. Equivalently, for each open $U\subset X$ the space $\mathcal K_S(U)$ consists of those tempered functions $f$ on $U$ with the property that for any Nash differential operator $D$ on $U\smallsetminus S$ the function $Df$ is bounded in a neighborhood of $S\cap U$. 
\end{lemma}

Before we prove the lemma, we mention a basic tool for our proofs, namely the embedded resolutions of singularities, in the following sense: 

\begin{theorem}\label{resolution} For every smooth semialgebraic set $X$ and a closed semialgebraic subset $S\subset X$, there is a smooth semialgebraic set $\tilde X$ and a proper morphism $p:\tilde X\to X$, such that:
\begin{enumerate}
 \item $p$ is an isomorphism away from $S$;
 \item there is a finite open cover $X=\bigcup U_i$ and, on each $U_i$, semialgebraic coordinates $(y_1,\dots, y_n)$ such that $S\cap U_i$ is given by finite intersections and unions of sets of the form: 
$$\{x|f(x)\ge 0\},$$
in the (real) archimedean case, and:
$$\{x|f(x)\in P_k\},\,\, P_k=\{y^k|y\in F\},$$
in the nonarchimedean case, where $f = c y_1^{i_1}\cdots y_n^{i_n}$, $i_j\in \mathbb N$.
\end{enumerate}
\end{theorem}
This follows from Hironaka's embedded resolution of singularities \cite[Corollary 3, p.\ 146]{Hi}.

Along with the previous lemma, we will also prove the following, which will be useful elsewhere:
\begin{lemma}\label{blowup}
 Consider a resolution $\tilde X\to X$ as in Theorem \ref{resolution}. Then for every open $U\subset X$ with preimage $\tilde U\subset \tilde X$, the pullback gives rise to an equality: 
$$\mathcal K_S(U) \simeq \mathcal K_{\tilde S}(\tilde U).$$
\end{lemma}

\begin{proof}[Proof of Lemmas \ref{lemmaKS} and \ref{blowup}]
We first prove that an element $f\in \mathcal O(U)$ belongs to all $\mathcal J_S^n(U)$ if and only if it vanishes on $S\cap U$ together with all its derivatives. One direction is easy: it is clear that any element of $\mathcal J_S^n(U)$ has vanishing $i$-th derivatives, for $i\le n$, on all points of $S$. 

Vice versa, consider a resolution $\tilde X\to X$ as in Theorem \ref{resolution}. If $\tilde U$ is the preimage of $U$, it is easy to see (by reduction to $\tilde U = F^n$ with $\tilde S= $ a ``standard'' semialgebraic set defined by conditions on the coordinate functions) that if a smooth function vanishes with all its derivatives on $\tilde S$, it coincides locally around each point of $\tilde S$ with the restriction of a Schwartz function on $\tilde U\smallsetminus\tilde S$. (We remind again that in the nonarchimedean case there are no higher derivatives, so the statement is about almost smooth functions vanishing on $\tilde S$.) Let $V$ denote the (closed) subspace of $\mathcal O(U)$ consisting of such functions; we will eventually prove that it coincides with $\mathcal K_{\tilde S}(\tilde U)$.

To do so, we use two well-known results in the archimedean case, which can be similarly be proven for the nonarchimedean case, for sets as in Theorem \ref{resolution}:
\begin{enumerate}
 \item For every Schwartz function $\phi$ on $\tilde U$ there is a positive, real-valued Schwartz function $\psi$ on the same space such that $\frac{\phi}{\psi}$ is a Schwartz function.
 \item Every Schwartz function $\phi$ on $\tilde U\smallsetminus \tilde S$ is a product of two Schwartz functions on the same space.
\end{enumerate}

From this it can immediately be deduced that any $f\in V$ can be written as a limit of $f\psi_\alpha$, where $\psi_\alpha$ runs over a suitable system of positive Schwartz functions on $\tilde U$, directed by majorization; and that every $f\psi_\alpha$ is the product of two Schwartz functions on $\tilde U\smallsetminus\tilde S$. But Schwartz functions on $\tilde U\smallsetminus\tilde S$ belong to $V$, hence the multiplication map: $V\otimes V\to V$ has dense image. Since $V \subset \mathcal J_{\tilde S}(\tilde U)$, it follows that $V\subset \mathcal J_{\tilde S}^2(\tilde U)$; repeating this argument, $V\subset \mathcal J_{\tilde S}^n(\tilde U)$ for all $n$. Hence $V\subset \mathcal K_{\tilde S}(\tilde U)$, therefore these spaces are equal.

Now, there is a sequence of natural numbers $k_n\to +\infty$ such that a smooth function on $\tilde U$ which vanishes with all its first $n$ derivatives on $\tilde S\cap \tilde U$ descends to a function on $U$ which vanishes with all its first $k_n$ derivatives on $S\cap U$ (vice versa, if the first $n$ derivatives of a function on $U$ vanish on $S$ then they also vanish for the pullback to $\tilde U$). This implies that elements of $\mathcal K_S(\tilde U)$ descend to smooth functions on $U$ with vanishing derivatives on $S$. Notice that the pullback of $\mathcal O(U)$ is closed in $\mathcal O(\tilde U)$, and the topologies coincide (indeed, since the map $\tilde X\to X$ is proper, it suffices for defining the topology on $\mathcal O(\tilde U)$ to consider only those Schwartz functions on $\tilde U$ which are pullbacks of Schwartz functions on $U$). 
Again by the fact that $\mathcal K_{\tilde S}(\tilde U)^2$ is dense in $\mathcal K_{\tilde S}(\tilde U)$ we deduce that $\mathcal K_{\tilde S}(\tilde U) \subset \mathcal J_S^n(U)$ for all $n$, hence $\mathcal K_{\tilde S}(\tilde U) = \mathcal K_S(U)$. 
\end{proof}

Now notice that we have an injective map: $\mathcal K_S(U)\to \mathcal O(U\smallsetminus S)$.

\begin{lemma}\label{lemmaexistsidentity}
 There is a sequence $(u_n)_n\subset \mathcal K_S(U)$ with: 
$$u_n\to 1 \mbox{ in }\mathcal O(U\smallsetminus S).$$
\end{lemma}

We will call such a sequence an \emph{approximate identity}, despite the fact that it is only bounded in the weaker topology of $\mathcal O(U\smallsetminus S)$, because it satisfies:
$$u_n f \to f\mbox{ for all }f\in \mathcal K_S(U).$$

\begin{proof}
We may choose a countable, increasing open cover $U_n$ of $U$ with the property that the complement of $U_n$ contains a neighborhood of $U\cap S$. Then we can find a sequence of tempered functions $u_n\in \mathcal K_S(U)$, with $u_n|_{U_n} \equiv 1$ and the property that for every Nash differential operator $D$ on $U\smallsetminus S$ there is a semialgebraic function $T$ on $U\smallsetminus S$ such that $|Du_n|\le |T|$ for all $n$. (Again, this is easier to see with a blowup as in Theorem \ref{resolution}.) By Lemma \ref{sequentialconvergence}, the limit of this sequence in $\mathcal O(U\smallsetminus S)$ is the constant function $1$.
\end{proof}

We denote the quotient sheaf $\mathcal O_S:= \mathcal O/\mathcal K_S$. It is supported on $S$. Notice that for every quotient of $\mathcal O$ by an ideal subsheaf $\mathcal I$, the map: $\mathcal O(U)\to (\mathcal O/\mathcal I)(U)$ is surjective for every $U$; thus, no sheafification is needed. Indeed, using a partition of unity as in \cite[Theorem 5.2.1]{AG-Nash}, we can patch functions $f_i$ on a finite cover $U=\bigcup_i U_i$, which agree on intersections modulo $\mathcal I$, to a function $f\in \mathcal O(U)$ whose restriction to $U_i$ is $\equiv f_i \mod \mathcal I(U_i)$. 

\begin{lemma}\label{temperedlemma} 
The sheaf $\mathcal O_S$  can also be described as the completion of $\mathcal O$ over the closed subset $S$: 
\begin{equation} \label{temperedlocalization}
 \mathcal O_S = \lim_{\underset{n}{\from}} \mathcal O/\mathcal J_S^n.
 \end{equation}
\end{lemma}

\begin{proof}
In the nonarchimedean case we have seen that $\mathcal K_S=\mathcal J_S$, so the statement is trivial. We discuss the archimedean case.

First of all, we use a resolution $\tilde X\to X$ as in Theorem \ref{resolution}. Then we have, for every $U$, a commutative diagram of injective maps:
\begin{equation*}
\xymatrix{
 \mathcal O/\mathcal K_S(U) \ar@{^{(}->}[r]\ar@{^{(}->}[d] & \lim_{\underset{n}{\from}} \mathcal O/\mathcal J_S^n(U)\ar@{^{(}->}[d] \\
 \mathcal O/\mathcal K_{\tilde S}(\tilde U) \ar@{^{(}->}[r]& \lim_{\underset{n}{\from}} \mathcal O/\mathcal J_{\tilde S}^n(\tilde U) 
} 
\end{equation*}

Let us fix a sequence $D_i$ of Nash differential operators which generate all Nash differential operators over the ring of Nash (i.e.\ smooth semialgebraic) functions, for the set $\tilde U$. An element of $\lim_{\underset{n}\from}\mathcal O/\mathcal J_{\tilde S}^n(\tilde U)$ can be described (non-uniquely) by a sequence $f_n\in \mathcal O(\tilde U)$ with the property: for all $i\le n$ we have $D_if_n|_{\tilde S} = D_if_i|_{\tilde S}$. The topology on $\lim_{\underset{n}\from}\mathcal O/\mathcal J_{\tilde S}^n(\tilde U)$ is given by seminorms:
$$ \sup_{x\in \tilde S} |T(x)D_i(\phi f_n)(x)|,$$
where $T$ varies over semialgebraic functions on $\tilde S$ and $\phi$ varies over elements of $\mathcal S(\tilde U)$.

It is then elementary to construct (by appropriate Taylor series) a smooth, tempered function $f$ on $\tilde U$ with $D_if|_{\tilde S} = D_if_i|_{\tilde S}$ for all $i$. Moreover, for every $\phi\in \mathcal S(U)$ the construction of $f$ can be made such that $\Vert D_i(f\phi)\Vert_{L^\infty(U)}$ is bounded in terms of a finite number of seminorms of the form: $$ \sup_{x\in \tilde S} |T(x)D_j(\phi_j f_n)(x)|,$$
where $T$ is semialgebraic and $\phi_j\in\mathcal S(\tilde U)$. Thus, the bottom horizontal arrow of the above diagram is an isomorphism.

But if the given element of $\mathcal O/\mathcal J_{\tilde S}^n(\tilde U) $ comes from $\mathcal O/\mathcal J_S^n(U)$, all the derivatives of the constructed function $f$ agree on $S$; therefore, it descends to an element of $\mathcal O(U)$. In other words, the vertical arrows are closed embeddings into the same subspace, and this proves the claim.
\end{proof}

\subsection{Schwartz cosheaves} 

By a \emph{Schwartz cosheaf} on $X$ we will mean a cosheaf $\mathcal F$ of nuclear Fr\'echet spaces on $X$, satisfying certain axioms. The extension maps will be denoted by $e_V^U$ (where $V\subset U$ are open subsets), or simply by $e$ when the source and target are clear. We will sometimes call the extension maps ``extension by zero'', to emphasize their geometric meaning. 
The axioms are expressed in terms of an arbitrary open subset $U$ (as I do not see a way to make them ``sheaf-theoretic'' combining the presence of a sheaf and a cosheaf), and are the following:

\begin{enumerate} 
 \item The extension maps $e_V^U: \mathcal F(V)\to \mathcal F(U)$ are closed.
 \item $\mathcal F(U)$ is a continuous $\mathcal O(U)$-module, i.e.\ there is a continuous bilinear map: $$\mathcal O(U) \times \mathcal F(U)\to \mathcal F(U)$$ compatible with multiplication on $\mathcal O(U)$.
 \item If $S\subset X$ is closed, sections of $\mathcal F$ vanishing to arbitrary degree on $S$ are extensions by zero of sections on the complement of $S$, that is: if $\mathcal J_S$ denotes the sheaf of ideals of tempered functions vanishing on $S$ as before then:
\begin{equation}\label{vanishing}
 \bigcap_n \overline{\mathcal J_S^n\mathcal F(U)}\subset e\left(\mathcal F(U\smallsetminus S)\right).
\end{equation}
(We will prove equality in the next lemma.)
 \item Obvious compatibility conditions: If $V\subset U$ are open then the diagram commutes:
$$\xymatrix{
\mathcal O(U)\ar[d]^{r^U_V}\ar@{}[r]|{\otimes} & \mathcal F(U)\ar[r] & \mathcal F(U) \\ 
\mathcal O(V) \ar@{}[r]|{\otimes} & \mathcal F(V)\ar[u]^{e^U_V}\ar[r] & \mathcal F(V) \ar[u]^{e^U_V}.
}$$
\end{enumerate}


\begin{lemma}\label{lemmaidentity}
 Let $(u_n)_n$ be a weak approximate identity in $\mathcal K_S(U)$ as in Lemma \ref{lemmaexistsidentity}, then $u_n f\to f$ for every $f\in e_V^U\left(\mathcal F(U\smallsetminus S)\right)$. Consequently:
$$ \bigcap_n \overline{\mathcal J_S^n\mathcal F(U)}= e\left(\mathcal F(U\smallsetminus S)\right) =  \overline{\mathcal K_S\mathcal F(U)}.$$
\end{lemma}

\begin{proof}
Let $V=U\smallsetminus S$. Since the map $\mathcal O(V)\times  \mathcal F(V)\to \mathcal F(V)$ is continuous, and $u_n\to 1$ in $\mathcal O(V)$, we have $u_nf\to f$ for every $f\in \mathcal F(V)$. By the compatibility of restrictions and corestrictions (fourth axiom), $u_n e_V^U(f)$ is the same as $e_V^U(u_n f)$, and since $e_V^U$ is continuous, this tends to $e_V^U(f)$ for all $f\in\mathcal F(V)$.

But $u_n e_V^U(f)\in \mathcal K_S(e_V^U \mathcal F(V))\subset \mathcal J_S^n\mathcal F(U)$ for all $n$, hence the claim.
\end{proof}

\subsection{Functoriality}

In what follows, we consider only morphisms between smooth semialgebraic sets. For a morphism $\pi:X\to Y$, and a cosheaf $\mathcal F$ on $X$, the push-forward $\pi_*\mathcal F$ is simply the cosheaf: $V\mapsto \mathcal F(\pi^{-1}V)$.

\begin{lemma}
 The push-forward of a Schwartz cosheaf is a Schwartz cosheaf.
\end{lemma}

\begin{proof}
 The only nontrivial verification is that of the third axiom. Let $\pi:X\to Y$ be a morphism and denote for clarity by $\mathcal O_X$ and $\mathcal O_Y$ the corresponding sheaves of tempered functions. Let $S\subset Y$ be a closed subset, and $\mathcal K_S, \mathcal K_{\pi^{-1}S}$ the corresponding sheaves on $Y$ and $X$, respectively.

 The third axiom for $\mathcal F$ implies that for every open  $V\subset Y$ we have an equality:
$$ \bigcap_n \overline{\mathcal J_{\pi^{-1}S}^n(\pi^{-1}V) \cdot  \pi_*\mathcal F(V)} = e\left(\pi_*\mathcal F(V\smallsetminus S)\right).$$
 In particular, since $\mathcal J_S(V) \subset \mathcal J_{\pi^{-1}S}(\pi^{-1}V)$, we get:
$$ \bigcap_n \overline{\mathcal J_S^n(V) \cdot  \pi_*\mathcal F(V)} \subset e\left(\pi_*\mathcal F(V\smallsetminus S)\right).$$

But by the proof of Lemma \ref{lemmaidentity}, we have: $$e\left(\pi_*\mathcal F(V\smallsetminus S)\right) \subset \overline{\mathcal K_S(V) \cdot  \pi_*\mathcal F(V)},$$ and of course $\overline{\mathcal K_S(V) \cdot  \pi_*\mathcal F(V)} \subset \bigcap_n \overline{\mathcal J_S^n(V) \cdot  \pi_*\mathcal F(V)}$. Hence, all these three spaces coincide.
\end{proof}

\subsection{Stalks, fibers and a Nakayama-type lemma} \label{ssstalks}

Given a closed $S\subset X$ we define the \emph{stalk} of a Schwartz cosheaf over $S$ to be the cosheaf:
$$\mathcal F_S = \mathcal F/e\left(\mathcal F_{X\smallsetminus S}\right),$$
where $\mathcal F_{X\smallsetminus S}$ is the cosheaf: $\mathcal F_{X\smallsetminus S}(U) = \mathcal F(U\smallsetminus S)$.
Clearly, $\mathcal F_S$ is zero on $X\smallsetminus S$; in that sense, it is supported on $S$.

Let now $\bar{\mathcal O}_S$ be the quotient $\mathcal O/\mathcal J_S$ -- it is the sheaf of restrictions to $S$ of (smooth) tempered functions on $\mathcal O$, and it satisfies $\bar{\mathcal O}_S (U) = \mathcal O(U)/\mathcal J_S(U)$. The \emph{fiber} of $\mathcal F$ will be the cosheaf: $\bar{\mathcal F}_S (U) = \mathcal F(U)/\overline{\mathcal J_S(U)\mathcal F(U)}$. Of course, in the nonarchimedean case the fiber and the stalk coincide. In the archimedean case, we will use the following version of Nakayama's lemma:

\begin{proposition}\label{propNakayama}
Let $S=$a point (hence $\bar{\mathcal O}_S(U)=\CC$ for every $U\supset S$), and assume that a \emph{finite-dimensional} subspace $N\subset \mathcal F(U)$ spans the fiber $\bar{\mathcal F}_S(U)$. Then the same subspace algebraically (i.e.\ without taking closures) generates $\mathcal F_S(U)$ as an $\mathcal O_S(U)$-module. 
\end{proposition}

\begin{proof}
For simplicity of notation, $M:=\mathcal F_S(U)$, $R = \mathcal O_S(U)$, $J=\mathcal J_S(U)$, $M_n=\overline{J^nM}$.  Notice that by the third axiom for Schwartz cosheaves, $M$ is separated with respect to the $J$-adic topology, i.e.\ we have an embedding:
$$M \hookrightarrow \lim_{\underset n\from} M/M_n.$$

We claim that for every $n$ the map $J^n\otimes N\to M_n/M_{n+1}$ is surjective. Indeed, we have a natural map with dense image:
$$ J^n/J^{n+1} \otimes_{R/J} N \to M_n/M_{n+1},$$
but since the space on the left is finite dimensional, the map is surjective. 

This implies that the map $(R/J^n) \otimes N\to M/M_n$ is surjective, and hence so is the map: $$\lim_{\underset n\from} (R/J^n) \otimes N \to \lim_{\underset n\from} M/M_n .$$
Since $N$ is finite-dimensional, the left hand side is equal to $N\otimes \lim_\from R/J^n$, which is equal to $N\otimes R$ by Lemma \ref{temperedlemma}. Therefore $N$ generates $\lim_{\underset n\from} M/M_n$ algebraically over $R$, and in particular $M=\lim_{\underset n\from} M/M_n$.
\end{proof}

\subsection{Group actions} 

Let $\mathcal F$ be a Schwartz cosheaf on a smooth $F$-variety $X$, and assume that it carries an action of a group $G$. (The group $G$ is assumed to act trivially on $X$; for example, if $Y$ is an affine $G$-variety, with $G$ reductive, $X=Y\sslash G$ and $\mathcal G$ is a cosheaf on $Y$ with a compatible $G$-action, then $\mathcal F$ could be the push-forward of $\mathcal G$.)

We let $\mathcal F_G$ denote the \emph{cosheaf of $G$-coinvariants}, that is:
$$\mathcal F_G(U) = \mathcal F(U)_G,$$
where $\mathcal F(U)_G$ denotes the quotient of $\mathcal F(U)$ by the closed subspace generated by vectors of the form $v-gv$, $v\in \mathcal F(U)$. 

\begin{proposition}\label{propcoinvisSchwartz}
 The cosheaf $\mathcal F_G$ is a Schwartz cosheaf. For every two open sets $V\subset U$, the map: $\left(e_V^U\mathcal F(V)\right)_G\to \mathcal F_G(U)$ is a closed embedding.
\end{proposition}

\begin{proof}
 The only nontrivial axiom to check is the closedness of extension maps, therefore it suffices to prove the second statement. Let $M_1=e_V^U\mathcal F(V)$, $M_2= \mathcal F(U)$, and let $N_i, i=1,2$, be the subspace of $M_i$ algebraically spanned by elements of the form $v-g\cdot v$, $v\in M_i$. We need to prove that for a sequence $f_i\to f$, where $f_i\in N_2$ and $f\in M_1$, we have $f\in \overline{N_1}$. 

 Choose an approximate identity $u_n \in \mathcal K_{U\smallsetminus V}(U)$ (Lemma \ref{lemmaexistsidentity}). Then we have: $\lim_n u_nf = f$ (Lemma \ref{lemmaidentity}). We also have $\lim_i u_n f_i = u_n f$ for every $i$; thus, $f \in \overline{(u_nf_i)_{n,i}}\subset \overline{N_1}$.
\end{proof}

In applications to the present paper, all Schwartz cosheaves that we will encounter are flabby, i.e.\ the extension maps are monomorphisms (hence closed embeddings, by the first axiom). The last proposition implies:
\begin{corollary}\label{corollaryflabby}
 In the above setting, the cosheaf of $G$-coinvariants of a flabby Schwartz cosheaf is flabby.
\end{corollary}

\end{appendices}

\bibliographystyle{alphaurl}
\bibliography{biblio}

\end{document}